\documentclass[a4paper,12pt,thmsa]{article}

\usepackage{amsfonts}
\usepackage{amssymb}
\usepackage{graphicx}
\usepackage{amsmath}
\usepackage{xr}
\usepackage{bezier}
\usepackage{bbm}
\usepackage{tikz}
\usepackage[normalem]{ulem}
\usepackage{dsfont}

\makeatletter

\renewcommand{\paragraph}{\@startsection{paragraph}{4}{0mm}{-3mm}{-3mm} {\noindent \bf}}

\renewcommand{\subparagraph}{\@startsection{paragraph}{4}{3mm}{-3mm}{-3mm} {\noindent \bf}}

\newcommand{\RR}{{\mathbb{R}}}

\newcommand{\EE}{{\mathbb{E}}}
\newcommand{\LL}{{\mathbb{L}}}

\newcommand{\PP}{{\mathbb{P}}}

\newcommand{\QQ}{{\mathbb{Q}}}

\newcommand{\CF}{{\mathcal{F}}}

\newcommand{\CA}{{\mathcal{A}}}
\newcommand{\CC}{{\mathcal{C}}}
\newcommand{\CI}{{\mathcal{I}}}
\newcommand{\CJ}{{\mathcal{J}}}

\newcommand{\CL}{{\mathcal{L}}}

\newcommand{\CB}{{\mathcal{B}}}

\newtheorem{lemma}{Lemma}
\newtheorem{proposition}{Proposition}

\newtheorem{example}{Example}
\newtheorem{Ass}{Assumption}

\def \ind{\mathbbm 1}

\begin{document}

\hyphenchar\font=-1

\baselineskip=7mm

\vspace{3cm}

\title{\textbf{A Class of Singular Control Problems \\ with Tipping Points}\thanks{This research has benefited from financial support of the ANR (Programme d'Investissements d'Avenir CHESS ANR-17-EURE-0010) and the research foundation TSE-Partnership (Chaire March\'es des Risques et Cr\'eation de Valeur, Fondation du Risque/SCOR).}}

\author{Jean-Paul D\'ecamps\thanks{Toulouse School of Economics, University of Toulouse Capitole, Toulouse, France. E-mail: \tt{jean-paul.}  \tt{decamps@tse-fr.eu}.}~~~~Fabien Gensbittel\thanks{Toulouse School of Economics, University of Toulouse Capitole, Toulouse, France. E-mail: \tt{fabien.} \tt{gensbittel@tse-fr.eu}.}\\Thomas Mariotti\thanks{Toulouse School of Economics, CNRS, University of Toulouse Capitole, Toulouse, France, CEPR, and CESifo. Email: \tt{thomas.mariotti@tse-fr.eu}.}~~~~St\'ephane Villeneuve\thanks{Toulouse School of Economics, University of Toulouse Capitole, Toulouse, France. Email: \tt{stephane.villeneuve@tse-fr.eu}.}} \vspace{1cm}

\maketitle \vspace*{6mm}

\begin{abstract}
Tipping points characterize situations where a regulated system may experience a sudden and irreversible change and are generally associated with a random state of the system below which the change materializes. In this paper, we study a singular stochastic control problem in which the performance criterion depends on the hitting time of a random state that is not a stopping time for the reference filtration. We establish a connection between the value of this problem and that of a singular control problem involving a diffusion and its running minimum. We provide a verification lemma that we apply to explicitly solve a resource-extraction problem with an ex-ante unknown tipping point.
\end{abstract}

\thispagestyle{empty}

\newpage
\setcounter{page}{1}

\section{Introduction}

Tipping points are typical of situations where a dynamically regulated system may undergo unpredictable irreversible changes. They are generally associated with a random state of the system below (or above) which the change materializes. There are numerous examples of tipping points in resource management and environmental sustainability. For instance, an increase in fires and grazing can exceed a forest's ability to regenerate, causing it to tip towards a savanna ecosystem. Excessive input of nutrients like nitrates and phosphates can lead to an explosive growth of algae, reducing oxygen in the water and creating dead zones (Hirota et al. \cite{Hirota}, Moore \cite{Moore}). The concept of tipping point has also entered the financial literature since the 2008 crisis, where too much home-loan defaults caused a decrease in the value of collateralized debt obligations, leading to the insolvency of banks and insurers (Battiston et al. \cite{Battiston}). Other examples include public health issues. For instance, in a pandemic, lock-down policies aim at controlling the number of infected people, balancing the risks of economic and health-system collapses.

A growing literature has developed models in which a tipping point is reached when a state variable of interest falls below a threshold whose exact value is ex-ante unknown, though it is distributed according to a known probability law. These models typically feature a decision maker (dm) who controls a state variable modeling the reserves of a valuable resource from which he can extract for his own consumption. Because the dm is uncertain about the value of the tipping point, he has to find the right balance between the benefits from current consumption and the risk of reaching the tipping point. Specifically, whereas an aggressive extraction policy raises current consumption and also brings valuable new information about the location of the tipping point each time a new minimum of the reserve process is reached, it also accelerates the time at which the system tips and thus risks over-exploiting and irreversibly damaging the resource.

To address the resulting optimal control problem, existing studies assume that the dynamics of reserves is deterministic conditional on the value of the tipping point and the dm's extraction behavior.\footnote{We review the relevant literature below.} Therefore, the only source of exogenous uncertainty is the level of the tipping point itself. However, real socio-ecological or economic systems are often subject to additional risks that are independent of the tipping point and lie beyond the control of the dm. These include, for instance, short-run variations in climatic conditions or contagion rates, fluctuations of demand or interest rates, or political instability. Incorporating such risks in tipping-point models is relevant and raises challenging questions for the characterization of optimal extraction behavior.

To address this problem, we consider a stylized resource-extraction model in which the uncontrolled reserve process is a diffusion $X$ and the extraction policy is modeled as a nondecreasing control process $L$, resulting in a controlled reserve process $X^L$. The tipping point is represented as an independent random level of reserves $Y$ that is ex-ante unknown to the dm, who only knows the distribution of $Y$. Extraction lasts until the first time $t$ such that $X^L_t \leq Y$. (Jumps in $L$ are allowed for, so that this inequality may be strict.) When this occurs, at time $\tau_Y$, the dm obtains a continuation value $U(X_{\tau_Y}^L)$, where $U$ can be interpreted as the value function of a downgraded extraction problem that the dm hopes to face as late as possible. The dm's objective is to maximize the expectation of the sum of an integral of his discounted extraction increments and of his continuation value,
\begin{align*}
{\EE} \hskip -0.5mm \left[\int_{[0,\tau_{Y}]} \mathrm e^{-rs} \, \mathrm dL_s + \mathrm e^{-r \tau_{ Y}}U( X_{\tau_{ Y}}^L) \right]\hskip -1mm.
\end{align*}
The dm's problem is thus a singular control problem with a random tipping point. The classic tradeoff between the costs and benefits of holding reserves now depends on the distribution of the tipping point and on the properties of the continuation value function, making this problem both novel and nontrivial.

Our first main result is that, because the time $\tau_Y$ when the system tips is the hitting time by $X^L$ of a region of the form $(- \infty, Y]$, the dm's problem admits a Markovian formulation in which the state variables are the current level of reserves and the minimum level they have reached so far. Intuitively, if the current level of reserves is above their historic minimum, then the dm momentarily does not risk crossing the tipping point, and he will learn nothing more about the location of the tipping point until reserves reach a new minimum level. Once $X^L$ and its running minimum $M^L$ have been identified as the relevant state variables, one is lead to a two-dimensional singular control problem in which the reward functional features the expectation of an integral with respect to future decrements of $M^L$. We provide a verification lemma for this new class of problems, taking care of the fact that a downward jump in the running minimum impacts differently the dm's reward functional depending on whether or not the tipping point has been crossed as a result of this jump.

Our second main result is to solve the dm's problem by means of an explicit construction of its value function and to characterize the optimal extraction policy when it exists. When reserves are at their historic minimum, three situations may arise. (1) Either there exists an optimal control process $L$ that allows the reserves to grow up to a free extraction boundary $b$, which is a function of the running minimum; we show that this free boundary is the solution to an ode which takes into account the law of the tipping point $Y$, the dynamics of the reserves before crossing the tipping point, and the continuation value function $U$. (2) Or it is optimal to extract in one go a lump-sum fraction of the reserves up to an endogenous level, after which the optimal control $L$ in (1) is applied. (3) Or no optimal control exists. This latter case corresponds to a situation in which extracting reserves up to an endogenous level in one go would be optimal, provided this did not result in crossing the tipping point. However, because the dm does not know the tipping point, he cannot implement such a policy. We show that, for each $\varepsilon >0$, there exists an $\varepsilon$-optimal policy, which consists in identifying the tipping point ``as fast as possible'' by extracting at a rate $\mathcal O\big({1 \over \varepsilon}\big)$.

\paragraph{Related Literature}

Our paper is related to a large literature in environmental economics, in which a tipping point is reached when an underlying state variable with deterministic dynamics crosses an unknown threshold. This approach, first considered in Kemp \cite{Kemp}, has been developed  and extended in several directions. For example, Tsur and Zemel \cite{Tsur1} study saltwater intrusion in groundwater resources, Tsur and Zemel \cite{Tsur2} and  Lemoine and Traeger \cite{Lemoine}  explore climate tipping points, N{\ae}vdal \cite{Naevdal} studies the eutrophication of lakes, Diekert \cite{Diekert} studies the cooperative or non-cooperative use of a resource under the threat of a disastrous threshold, and Guillouet and Martimort \cite{Guillouet} and Liski and Salani\'e \cite{Liski} develop models with delay between the unobservable crossing of a tipping point and the observable occurrence of a catastrophe. We provide this literature with a rigorous mathematical framework that enables us to account for a stochastic dynamics of the state variable.

Our paper is related to the literature on the optimal management of a resource in diffusion settings. The word \textit{resource} is here meant in a broad sense. It can refer to corporate cash, as in Jeanblanc-Picqu\'e and Shiryaev \cite{Jeanblanc}, Radner and Shepp \cite{Radner}, or D\'ecamps, Mariotti, Rochet, and Villeneuve \cite{Dec1}, populations, as in Alvarez and Hening \cite{Alvarez1}, or any asset subject to a control problem of storage or inventory type, as in Shreve, Lehoczky, and Gaver \cite{SGL}, which serves as the natural benchmark for our model. These studies develop one-dimensional singular stochastic control problems, in which the optimal policy is to extract the resource when its reserves exceed an endogenous critical value.

Other studies have incorporated stochastic profitability or stochastic interest rates into the standard optimal management model. These extensions led to new and challenging control problems in two-dimensional diffusion settings, as in Reppen, Rochet, and Soner \cite{Reppen}, De Angelis \cite{Deangelis1},  Bandini, De Angelis, Ferrari, and Gozzi \cite{Bandini}, and D\'ecamps and Villeneuve \cite{Dec2}. The present paper deals with a different type of  problem, in which the dm controls the level of the resource and its running minimum. While there exists an extensive literature on two-dimensional optimal stopping problems involving the running maximum (or the running minimum) of a one-dimensional diffusion,\footnote{Shepp and Shiryaev \cite{Shepp}, Dubins, Shepp, and Shiryaev \cite{Dubins}, Graversen and Peskir \cite{Graversen}, and Peskir \cite{Peskir} offer seminal contributions. See also Guo and Zervos \cite{Guo}, Ott \cite{Ott}, Gapeev and Rodosthenous \cite{Gapeev}, Rodosthenous and Zervos \cite{Rodos}, and D\'ecamps, Gensbittel, and Mariotti \cite{Dec3} for more recent contributions.} stochastic control problems with similar features have been scarcely ever considered.

To the best of our knowledge, the contemporaneous work of Ferrari and Rodosthenous \cite{FerrariNeofytos} is the only other study that develops a two-dimensional optimal control problem involving the running minimum of a diffusion process. Our paper and \cite{FerrariNeofytos} differ and complement each other in several ways. First, in \cite{FerrariNeofytos}, the class of reward functionals is chosen so that a standard Neumann condition expresses the behavior of the value function on the diagonal $X^L = M^L$; using their terminology, the class of reward functionals is consistent with the Hamilton--Jacobi--Bellman (HJB) equation for the resulting two-dimensional singular control problem. This property is not satisfied in our setting, where the reward functional is not set a priori but follows from the explicit modeling of a tipping point, a novel feature of the present paper. Second, within our class of problems, jumps of the process $(X^L, M^L)$ impact the reward functional differently and, as a result, the construction of candidate value functions raises other difficulties and requires a new verification lemma. Accordingly, the tipping-point problem that we solve analytically in Section 4 presents a number of unique features---most notably, an integral with respect to the controlled minimum process; in the Markovian formulation of the tipping-point problem, this reflects the fact that the dm does not know at which level of the controlled minimum process the system will tip. Third, depending on the parameters of the model, our tipping-point problem admits optimal or only $\varepsilon$-optimal strategies, which we characterize analytically. As far as we know, all our results, including the very class of two-dimensional control problems we study, are novel in the field of stochastic control and its applications.

\vskip 3mm

The paper is organized as follows. Section 2 describes our tipping-point problem. Section 3 provides the Markovian formulation of this problem, presents the HJB equation, and proves the required verification lemma. Section 4 characterizes the ($\varepsilon$-) optimal extraction policy. Section 5 concludes.

\section{Problem Formulation}

\paragraph{The Uncontrolled Reserve Process}

Let $X: =(X_t)_{t \ge 0}$ be a linear time-homogeneous diffusion defined on a complete probability space $(\Omega, \CF, \PP)$ that is a strong solution to
\begin{align*}
\mathrm dX_t =  \mu(X_t) \, \mathrm dt  +  \sigma(X_t) \, \mathrm dB_t, \quad X_0=x,
\end{align*}
where $B:=(B_t)_{t \ge 0}$ is a one-dimensional Brownian motion defined on $\Omega$. The state space for $X$ is an interval $\CI=(\alpha, \infty)$, with $-\infty \le \alpha < 0$. The following assumption is maintained throughout the paper.

\begin{Ass}\label{coefficient}
The functions $\mu$ and $\sigma$ \pagebreak are Lipschitz on $\CI,$ $\sigma> 0$ on $\CI,$ and $\alpha$ and $\infty$ are natural boundaries for $X$.
\end{Ass}

We let $\mathbb F := (\CF_t)_{t\ge 0}$ be the augmented right-continuous filtration generated by $B$ on $\Omega$.

\paragraph{Controls}

Let $\LL$ be the set of $\mathbb F$-adapted processes $L : =(L_t)_{t \geq 0}$ that are $\PP$-a.s.\! nondecreasing and right-continuous, with $L_{0-} =0$. We consider the controlled reserve process $X^L:=(X^L_t)_{t \geq 0}$ on $(\Omega, \CF, \mathbb F,\PP)$ satisfying
\begin{align}\label{eq1}
\mathrm dX_t^{L} =  \mu(X_t^{L}) \, \mathrm dt + \sigma(X_t^{L}) \, \mathrm dB_t-\mathrm dL_t, \quad X_{0-}^L=x.
\end{align}
Thanks to Assumption \ref{coefficient}, the sde \eqref{eq1} admits a unique strong solution for all $L \in \LL$ (see \cite[Chapter V, \S3, Theorem 7]{ProtterBook});\footnote{Precisely, one may assume that $\mu$ and $\sigma$ are Lipschitz functions on $\RR$ by taking any Lipschitz extension to obtain a solution $X^L$ taking values in $\RR$. Thereafter, we only consider solutions taking values in $\CI$.} clearly, $X^L \equiv X$
if $L\equiv 0$. As it is reasonable to assume that the controlled reserve process $X^L$ cannot take negative values, we restrict the class of controls accordingly. Specifically, letting $(\cdot)^+$ denote the positive-part function on $\mathbb R$, the class of admissible controls with initial condition $X^L_{0-}=x$ is
\begin{align}
\CA(x): =\big\{L \in \LL :  X^L_t \in \CI \text{ and } (X_{t-}^L)^+-(\Delta L)_t \ge 0 \text{ for all } t \ge 0\big\}. \label{defsetcontrols}
\end{align}
Accordingly, we may assume without loss of generality that $x \geq 0$. Thus, for each $L \in \CA(x)$, letting $\tau_y : =\inf\{t \geq 0: X_t^L \leq y\}$ for $y\in \CI$, we have $X^L_t \geq 0$ for all $t\leq\tau_0$.

\paragraph{Tipping Points}

The tipping point is represented by a random variable $Y$ with law ${\QQ}$, taking values in $\CI$, and independent of $B$. We assume that the tipping point will be reached before the resource is depleted, so that $Y$ is nonnegative, and also, for technical reasons, that $Y$ admits a density.

\begin{Ass}\label{density}
The cdf $F$ of $Y$ has a density $f$ on $\CI$ with respect to Lebesgue measure$,$ and there exists $\overline y>0$ such that $f$ is Lipschitz and positive on $[0,\overline y]$ and nil on $\CI\setminus[0,\overline y]$.
\end{Ass}

For each $L \in \mathcal A(x)$, $X^L$ is $\mathbb F$-adapted and thus independent of $Y$. It will be convenient to work with the product probability space $(\overline \Omega, \overline \CF, \overline {\PP}) := (\Omega \times \CI, \CF \otimes \CB(\CI), \PP \otimes \QQ)$, where $\CB(\CI)$ is the Borel $\sigma$-field on $\CI$, and we will assume that $Y(\omega,y)=y$ is defined on $\overline\Omega$. We denote by $\EE$ and $\overline {\EE}$ the expectation operators associated to $\PP$ and $\overline {\PP}$, respectively. Random variables defined on $\Omega$ are identified to variables defined on $\overline\Omega$. The random time $\tau_Y$ at which the tipping point is reached is defined by
\begin{align}
\label{tippingpoint}
\tau_Y: = \inf \hskip 0.5mm  \{ t \ge 0 : X_t^L \le Y \}.
\end{align}
Clearly, $\tau_Y$ depends on $L$, but we omit this dependence to simplify notation. To emphasize the dependence of $X^L$ on $x$, we will use the slightly abusive notation $\PP_{x}:=\PP \hskip 0.3mm [\cdot \! \mid \!X^L_{0-}=x]$ and $\EE_{x}:=\EE \hskip 0.3mm [\cdot \! \mid \! X^L_{0-}=x]$, with $\overline\PP_x$ and $\overline\EE_x$ defined similarly. In what follows, we use the convention that
\begin{align*}
\text{on } \{\tau=\infty\}, \; \int_{[0,\tau]}H_s \, \mathrm dA_s := \int_{[0,\infty)} H_s \, \mathrm dA_s \text{ and } \mathrm e^{-r\tau} H_\tau:=0
\end{align*}
for any random time $\tau: \overline\Omega \rightarrow [0,\infty]$, any measurable process $H$ (not necessarily defined at time $\infty$), and any process with finite variation $A$.

\paragraph{The Tipping-Point Problem}

The resource-extraction problem when there is a risk of reaching a tipping point, at which the dm receives a continuation value $U:\CI \mapsto \RR_+$, can now be defined as follows:
\begin{align}
\overline V(x) : = \sup_{L \in {\CA(x)}}  \overline{\EE}_x \! \left[\int_{[0,\tau_{Y}]} \mathrm e^{-rs} \, \mathrm dL_s + \mathrm e^{-r \tau_{ Y}}U( X_{\tau_{ Y}}^L) \right]\hskip -1mm, \label{pb}
\end{align}
for some discount rate $r >0$.

\section{Dynamic Programming}

In this section, we show how to reduce the tipping-point problem \eqref{pb} to a Markovian singular control problem, and we provide a key verification lemma.

\subsection{A Markovian Formulation}

Given $L \in \LL$ and $m \in [0, x]$, define the minimum process $M^L : = (M^L_t)_{t \geq 0}$ by $M^L_t : = m \wedge\inf_{s \leq t} X^L_s$. The pair $(X^L,M^L)$ defines an $\mathbb F$-adapted process such that $(X^L_{0-},M^L_{0-})=(x,m)$. To emphasize the dependence on $(x,m)$, we will use the notation $\PP_{x,m}:=\PP\hskip 0.3mm [\cdot \! \mid \! (X^L_{0-},M^L_{0-})=(x,m)$ and $\EE_{x,m}:=\EE \hskip 0.3mm [\cdot \! \mid \! (X ^L _{0-},M^L_{0-})=(x,m)]$. As we only consider the process $(X^L,M^L)$ up to the hitting time $\tau_0$, the natural state space for our problem is
\begin{align*}
\CJ: =\{(x,m) \in (0,\infty)^2 :  x\ge m \}.
\end{align*}
Our first result relates the tipping-point problem to a singular stochastic control problem for a one-dimensional diffusion and its running minimum.

\begin{proposition} \label{markovprop}
For each $(x,m)\in \CJ,$ consider the singular control problem
\begin{align} \label{markovv}
V(x,m) : =  \sup_{L \in {\cal A}(x)} V(x,m;L),
\end{align}
with
\begin{align}
V(x,m;L)& := \EE_{x,m} \!\left [ \int_{[0,\tau_0]} \mathrm e^{-rs}  F(M^L_{s-})  \, \mathrm dL_s \right] \! -\EE_{x,m} \! \left [\int_{[0,\tau_0]} \mathrm e^{-rs} U(M^L_{s}) f(M^L_{s}) \, \mathrm dM^{L,c}_s \right ] \nonumber
\\
& \quad \; +   \EE_{x,m} \!\left [\sum_{0\leq s \leq \tau_0} \mathrm e^{-rs}  U(M^L_{s}) [F(M^L_{s-}) - F({M^L_s})] \right]\hskip -1mm,  \label{reward}
\end{align}
where $M^{L,c}$ is the continuous part of $M^L$. Then$,$
\begin{align} \label{haha}
\overline V(x) = V(x,x) + U(x)[1 - F(x)] .
\end{align}
\end{proposition}

\noindent\textbf{Proof.} First, notice that $\tau_{ Y} = 0$ on $\{ Y \geq x \}$. Thus, for each $L \in \CA(x)$,
\begin{align*}
\overline{\EE}_x \! \left  [\int_{[0,\tau_{Y}]} \mathrm e^{-rs} \, \mathrm dL_s + \mathrm e^{-r \tau_{ Y}}U( X_{\tau_{ Y}}^L) \right] \!= J(x,L)  +  U(x) [1 - F(x)],
\end{align*}
where
\begin{align*}
J(x,L):= \overline{ \EE}_x \! \left [\left [\int_{[0,{\tau_{Y} }]} \mathrm e^{-rs}\,\mathrm dL_s + \mathrm e^{-r \tau_{ Y}}U( X^L_{\tau_{ Y}})\right] \! \ind_{Y < x}\right ] \hskip -1mm.
\end{align*}
Then, considering the family of $\mathbb F$-stopping times $(\tau_y)_{y \in \CI}$, and letting $M^L$ be defined with initial condition $M^L_{0-}=x$, we have, by independence of $B$ and $Y$,
\begin{align}
J(x,L)& =  \EE_x \! \left [\int_0^x \! \left[\int_{[0,{\tau_y}]}\mathrm e^{-rs} \, \mathrm dL_s + \mathrm e^{-r\tau_y} U(X^L_{\tau_y}) \right ]\! f(y) \, \mathrm dy \right ] \nonumber
\\
& =  \EE_x \! \left [\int_0^x \! \left[ \int_{[0,\infty)}\mathrm e^{-rs} \ind_{s \leq\tau_y} \, \mathrm dL_s   + \mathrm e^{-r\tau_y} U(X^L_{\tau_y})\right ]\! f(y) \, \mathrm dy \right ] \nonumber
\\
& =  \EE_{x} \! \left [ \int_{[0,\infty)} \mathrm e^{-rs} \! \left(\int_0^x \ind_{s \leq\tau_y}f(y) \, \mathrm dy \right)  \mathrm dL_s   +  \int_0^x \mathrm e^{-r \tau_y} U(X^L_{\tau_y}) f(y) \,\mathrm dy \right ] \nonumber
\\
& =   \EE_{x,x} \! \left [ \int_{[0,\infty)} \mathrm e^{-rs} \! \left ( \int_0^{ M^L_{s-}} f(y) \, \mathrm dy \right ) \mathrm dL_s + \int_0^x \mathrm e^{-r \tau_y} U(X^L_{\tau_y}) f(y) \, \mathrm dy\right] \nonumber
\\
& =   \EE_{x,x} \! \left [ \int_{[0,\infty)} \mathrm e^{-rs} F(M^L_{s-})  \, \mathrm dL_s  + \int_0^x \mathrm e^{-r \tau_y} U(M^L_{\tau_y}) f(y) \,\mathrm dy\right]\hskip -1mm, \label{deriv'}
\end{align}
where the third equality follows from Fubini's theorem and the fourth equality follows from the relation
\begin{align*}
\{M^L_{s-} >  y \} \subset \{s \leq \tau_y \} \subset  \{M^L_{s-} \geq  y \}.
\end{align*}
We now show that, for each $m\leq x$,
\begin{align}
& \EE_{x,m}\! \left [ \int_0^m \mathrm e^{-r\tau_y}  U(M^L_{\tau_y}) f(y) \,\mathrm dy \right ] \label{inter}
\\
&= - \,  \EE_{x,m} \!\left [ \int_{[0, \tau_0]}\mathrm e^{-rs}  U(M^L_{s}) f(M^L_s )\, \mathrm dM^{L,c}_s \right ]\! + \EE_{x,m}\!\left [\sum_{0 \leq s \leq \tau_0}\mathrm e^{-rs} U(M^L_s) [F(M^L_{s-}) - F(M^L_s)] \right]\hskip -1mm. \nonumber
\end{align}
Because $f(y)-f(M^L_{\tau_y})$ is different from $0$ only if $M^L$ has a jump at ${\tau_y}$, we have
\begin{align*}
&  \EE_{x,m} \! \left [ \int_0^m \mathrm e^{-r\tau_y}  U(M^L_{\tau_y}) f(y) \, \mathrm dy \right ]
\\
&=  \EE_{x,m}\! \left [ \int_0^m \mathrm e^{-r\tau_y}  U(M^L_{\tau_y}) f(M^L_{\tau_y} ) \, \mathrm dy \right ] \! +  \EE_{x,m} \!\left [ \int_0^m\mathrm e^{-r\tau_y} U( M^L_{\tau_y}) [f(y) -f(M^L_{\tau_y} )] \, \mathrm dy\right] \nonumber
\\
&=  \EE_{x,m} \! \left [ \int_0^m \mathrm e^{-r\tau_y}  U(M^L_{\tau_y}) f(M^L_{\tau_y} ) \, \mathrm dy \right ] \! +  \EE_{x,m} \! \left [ \sum_{0\leq s \leq \tau_0} \int_{M^L_s}^{M^L_{s^-}} \mathrm e^{-r\tau_y} U( M^L_{\tau_y})
[f(y) -f(M^L_{\tau_y} )] \, \mathrm dy\right]\hskip -1mm .
\end{align*}
For the first term of the last equality, a standard change of variables formula for Stieltjes integrals (\cite[Chapter 0, Proposition 4.9]{RevuzYor}) yields
\begin{align*}
\EE_{x,m} \!\left [ \int_0^m \mathrm e^{-r\tau_y}  U(M^L_{\tau_y}) f(M^L_{\tau_y} ) \, \mathrm dy \right ] =  - \, \EE_{x,m} \!\left [ \int_{[0, \tau_0]} \mathrm e^{-rs}  U(M^L_{s}) f(M^L_s )\, \mathrm dM^L_s \right]\hskip -1mm .
\end{align*}
For the second term, observe that $\tau_y=s$ and $M^L_{\tau_y}=M^L_s$ for all $y \in [M^L_s,M^L_{s-})$. Therefore,
\begin{align*}
&\EE_{x,m} \! \left [ \sum_{0\leq s \leq \tau_0} \int_{M^L_s}^{M^L_{s-}} \mathrm e^{-r\tau_y} U( M^L_{\tau_y}) [f(y) -f(M^L_{\tau_y} )]\, \mathrm dy\right]
\\
&= \EE_{x,m}\! \left [\sum_{0\leq s \leq \tau_0} \mathrm e^{-rs} U(M^L_s )\int_{M^L_s}^{M^L_{s-}}[f(y) - f(M^L_s)] \, \mathrm dy \right]
\\
&= \EE_{x,m}\! \left [\sum_{0\leq s \leq \tau_0} \mathrm e^{-rs} U(M^L_s) [F(M^L_{s-}) - F(M^L_s) - f(M^L_s) (M^L_{s-} - M^L_s)]\right] \hskip -1mm ,
\end{align*}
and thus
\begin{align*}
& \EE_{x,m} \!\left [ \int_0^m \mathrm e^{-r\tau_y}  U(M^L_{\tau_y}) f(y) \, \mathrm dy \right ]
\\
&=  - \, \EE_{x,m} \! \left [ \int_{[0, \tau_0]} \mathrm e^{-rs}  U(M^L_{s}) f(M^L_s ) \, \mathrm dM^L_s \right ]
\\
&  \quad \hskip 0.2mm + \EE_{x,m} \!\left [\sum_{0\leq s \leq \tau_0}\mathrm e^{-rs} U(M^L_s) [F(M^L_{s-}) - F(M^L_s) - f(M^L_s) (M^L_{s-} - M^L_s)]\right]
\\
& =  - \, \EE_{x,m}\! \left [ \int_{[0, \tau_0]}\mathrm e^{-rs}  U(M^L_{s}) f(M^L_s ) \, \mathrm dM^{L,c}_s\right ]\! - \EE_{x,m}\! \left [\sum_{0\leq s \leq \tau_0} \mathrm e^{-rs} U(M^L_s) f(M^L_s) (M^L_s - M^L_{s-})\right ]\allowdisplaybreaks
\\
& \quad \hskip 0.2mm + \EE_{x,m}\! \left [\sum_{0\leq s \leq \tau_0}\mathrm e^{-rs} U(M^L_s) [F(M^L_{s-}) - F(M^L_s) - f(M^L_s) (M^L_{s-} - M^L_s)]\right]
\\
& = - \,\EE_{x,m}\!\left[\int_{[0, \tau_0]}\mathrm e^{-rs}  U(M^L_{s}) f(M^L_s )\, \mathrm dM^{L,c}_s \right ]\! + \EE_{x,m}\! \left[\sum_{0\leq s \leq \tau_0}\mathrm e^{-rs} U(M^L_s) [F(M^L_{s-}) - F(M^L_s)]\right] \hskip -1mm ,
\end{align*}
which is (\ref{inter}). From (\ref{deriv'})--(\ref{inter}) with $m=x$, we have $J(x,L)=V(x,x;L)$. Hence the result. \hfill $\blacksquare$

\bigskip

The interpretation of (\ref{haha}) is that $U(x)$ corresponds to the dm's value function conditional on the tipping point being above $x$, which occurs with probability $1 - F(x)$, whereas the ratio $V(x,x)\over F(x)$ corresponds to the dm's value function conditional on the tipping point being below $x$, which occurs with probability $F(x)$.

\paragraph{Comparison with \cite{FerrariNeofytos}}

Proposition \ref{markovprop} shows that computing $V(x,x)$ involves solving the singular control problem (\ref{markovv}) for the two-dimensional process $(X^L,M^L)$ on the state space ${\mathcal J}$. This problem differs from the one studied in \cite{FerrariNeofytos} because jumps of the process $(X^L,M^L)$ impact the reward functional in a different way. Using the integral operators $\diamond$ and $\scriptstyle \square$ introduced in \cite[Section 3]{FerrariNeofytos}, the reward functional in \cite{FerrariNeofytos} for our model writes as
\begin{align}
\widehat V(x,m;L) &:= \EE_{x,m}\! \left[\int_{[0,\tau_0]} \mathrm e^{-rs} [F(M_s^L) \diamond \mathrm dL_s -U(M_s^L) f(M_s^L) \,{\scriptstyle \square} \, \mathrm dM_s^L ] \right] \notag
\\
&= \EE_{x,m} \!\left [ \int_{[0,\tau_0]} \mathrm e^{-rs}  F(M^L_{s-})  \, \mathrm dL_s \right] \! -\EE_{x,m} \! \left [\int_{[0,\tau_0]} \mathrm e^{-rs} U(M^L_{s}) f(M^L_{s}) \, \mathrm dM^{L,c}_s \right ] \notag
\\
& \quad + \EE_{x,m} \!\left [\sum_{0\leq s \leq \tau_0} \mathrm e^{-rs}\! \left[\int_{M_{s}^L}^{M_{s^-}^L}[U( u) f(u)   + F(u)] \, \mathrm du - (  M_{s^-}^L-M_{s}^L) F(M_{s^-}^L) \right] \right]\hskip -1mm, \notag
\end{align}
In case $x=m$, using similar arguments as in the proof of Proposition \ref{markovprop}, we can show that
\begin{equation}
\widehat V(x,x;L)= \overline{\EE}_x \! \left[ \left\{\int_{[0,\tau_{Y}]} \mathrm e^{-rs} \, \mathrm dL_s + \mathrm e^{-r \tau_{ Y}}[ U(Y)-(Y-X_{\tau_Y}^L)] \right\} \hskip -0.2mm \ind_{Y<x}  \right]\hskip -1mm, \label{formule}
\end{equation}
which coincides with \eqref{pb} on $\{Y<x\}$ whenever $U(m) \equiv m$ (up to an additive constant). Intuitively, if this condition is satisfied in our tipping-point problem, it does not matter if an upward jump in $L$ causes $X^L$ (and thus $M^L$) to cross the tipping point $Y$ discontinuously, because the marginal values $U'(X^L_{\tau_Y-})$ and $U'(X^L_{\tau_Y})$ of the reserves before and after the jump are both equal to 1, the marginal value of an increment of extraction. By contrast, if, say, $U$ is concave and $U'(X^L_{\tau_Y}) > U'(X^L_{\tau_Y-}) \geq 1$, then, in our formulation of the problem, the dm is exposed to a risk that is absent in \cite{FerrariNeofytos}, because the continuation payoff at the tipping point is equal to $U(X^L_{\tau_Y})$ in \eqref{pb}, whereas it is equal to $U(Y) - (Y - X^L_{\tau_Y}) > U(X^L_{\tau_Y})$ in \eqref{formule}.

\subsection{A Verification Lemma}

Let $\CL$ be the infinitesimal generator of the uncontrolled diffusion $X$, acting on functions on $\CJ$ that are twice differentiable with respect to $x$:
\begin{align*}
\CL h (x,m) := \mu(x) \, \frac{\partial h}{\partial x}\,(x,m)+ \frac{1}{2}\, \sigma^2(x) \, \frac{\partial^2 h}{\partial x^2} \, (x,m), \quad (x,m)  \in  \mathcal J.
\end{align*}
By the dynamic-programming principle, we expect the value function $V$ of \eqref{markovv} to satisfy, in some sense, the HJB equation
\begin{align}
\label{HJB}
\max\big\{(\CL -r)V(x,m),F(m)-V_x(x,m) \big\}=0
\end{align}
subject to the boundary conditions
\begin{align}
V_m(m,m) &\ge U(m) f(m),  \label{diagonal1}
\\
V(m,m) &\ge T[V](m),  \label{diagonal2}
\end{align}
where the operator $T$ is defined on the set of continuous functions on $\CJ$ by
\begin{align}
\label{operatorT}
T[V](m)=\sup_{0\le h \le m} \big\{hF(m)+U(m-h)[F(m)-F(m-h)]+V(m-h,m-h) \big\}.
\end{align}
Condition (\ref{diagonal1}) is a standard Neumann condition which states that, starting from a point $(m,m)$ on the diagonal $\partial \CJ:=\{ (x,m) \in \CJ:  x=m\}$, a marginal decrease in the second argument yields at least an expected payoff $U(m) f(m)$, reflecting that $f(m) $ corresponds to the probability that the tipping point lies in the infinitesimal interval $[m- \mathrm dm, m]$. Condition (\ref{diagonal2}) is a special feature of our class of control problems. To grasp the intuition for (\ref{diagonal2}), consider a couple $(m,m)$ and choose a control such that $L_0=h>0$. This control must yield no more than the optimal control. Using Proposition \ref{markovprop} together with a dynamic-programming argument, we must thus have, for each $h \in(0,m]$,
\begin{align}  \label{geneneuman}
V(m,m) \geq  h F(m)+U(m-h)[F(m)-F(m-h)] + V(m-h,m-h).
\end{align}
The HJB equation (\ref{HJB}) and the Neumann condition (\ref{diagonal1}) imply that (\ref{geneneuman}) is satisfied at the limit when $h\to 0$.\footnote{To see this, simply note that (\ref{geneneuman}) is equivalent to $\frac{V(m,m) - V(m-h,m-h)}{h} \geq F(m) + U(m-h) \, \frac{F(m) - F(m-h)}{h}$. Taking the limit as $h \to 0$ yields $V_x(m,m) + V_m(m,m) \geq F(m) + U(m) f(m)$, which is indeed satisfied because $V_x(m,m) \geq F(m)$ and $V_m(m,m) \geq U(m) f(m)$ by (\ref{HJB})--(\ref{diagonal1}).} However, constructing a candidate solution for the singular control problem \eqref{markovv} also requires checking that (\ref{geneneuman}) is satisfied for all $h\in(0,m]$, thus introducing the non-local operator $T$ in the boundary condition (\ref{diagonal2}).

\paragraph{Remark}

Observe that the reverse inequality $V(m,m) \leq T[V](m) $ always holds by definition of $T[V](m)$ (substitute $h=0$ in the right-hand side of (\ref{diagonal2})). Thus we expect \eqref{diagonal1} to hold as an equality everywhere.

\vskip 3mm

Notice that, if indeed $V_x(x,m) \ge F(m)$ on $\CJ$ as prescribed by \eqref{HJB}, we have
\begin{align}
\label{operatorT-2}
V(x,m) \ge (x-x')F(m) +V(x',m) \text{ for all } x\geq x' \geq m > 0.
\end{align}
We now present our verification lemma, which is based on applying It\^o's formula for the processes $(X^L_t,M^L_t)$ evolving on the \textit{closed} set $\CJ$. Justifying It\^o's formula on a closed set requires precisely defining some class of functions, which we subsequently denote by $\mathcal {R}(\CJ)$, to which it applies. To streamline the presentation, we refer the reader to Appendix \ref{Appendix:class} for a rigorous definition of this class.

\begin{lemma}\label{verif}
If $w\in \mathcal R(\CJ)$ is a nonnegative solution to the HJB equation \eqref{HJB} with boundary conditions \eqref{diagonal1}--\eqref{diagonal2}$,$ then $w \ge V$ on $\CJ$.
\end{lemma}

\noindent \textbf{Proof.} For any control $L \in \CA(x)$, we use a standard localization procedure (\cite[Chapter 1, \S5, Proof of Theorem 5.8]{KaratzasShreve}) by introducing the increasing sequence of stopping times
\begin{align} \label{localiz}
T_n:=\inf\left\{ t \ge 0 : \int_0^{t\wedge \tau_0} \mathrm e^{-2rs} \sigma^2(X_s^L)w^2_x(X_s^L,M_s^L)\, \mathrm ds \ge n \right\}
\end{align}
to obtain that the process
\begin{align*}
\left(\int_0^{t\wedge \tau_0\wedge T_n} \mathrm e^{-rs} \sigma(X_s^L)w_x(X_s^L,M_s^L)\, \mathrm d B_s \right)_{t \ge 0}
\end{align*}
is a martingale. Applying It\^{o}'s formula and  taking expectations, we obtain
\begin{align}
w(x,m) & =  \EE_{x,m}\! \left[ \mathrm e^{-r (t\wedge \tau_0 \wedge T_n)} w(X^L_{t\wedge \tau_0\wedge T_n}, M^L_{t\wedge \tau_0 \wedge T_n}) \right] \notag
\\
& \quad  -  \EE_{x,m} \! \left[\int_0^{t\wedge \tau_0 \wedge T_n} \mathrm e^{-rs} ({\cal L}-r)w(X^L_{s},M^L_{s}) \, \mathrm ds \right] \notag
\\
& \quad - \EE_{x,m}\! \left[\int_{[0, {t\wedge \tau_0\wedge T_n}]} \mathrm e^{-rs} w_m (M^L_{s}, M^L_{s})\,\mathrm dM_s^{L,c} \right] \notag
\\
& \quad +\EE_{x,m}\! \left[ \int_{[0, {t\wedge \tau_0\wedge T_n}]}\mathrm e^{-rs} w_x (X^L_{s}, M^L_{s})\,\mathrm dL_s^c \right] \notag
\\
& \quad -  \EE_{x,m} \! \left[ \sum_{0 \le s \le t\wedge \tau_0 \wedge T_n} \mathrm e^{-rs} [w(X^L_s, M^L_s) - w(X^L_{s-} , M^L_{s-})] \right] \hskip -1mm, \label{ineq0}
\end{align}
where $M^{L,c}$ and $L^c$ denote the continuous parts of $M^L$ and $L$, respectively. Using \eqref{HJB}--\eqref{diagonal2} and observing that $w$ is nonnegative and $-\,\mathrm dM_s^{L,c}$ is a positive measure, we obtain
\begin{align*}
w(x,m)  \ge & -\EE_{x,m} \! \left[\int_{[0, {t\wedge \tau_0\wedge T_n}]} \mathrm e^{-rs} U(M^L_{s})f(M^L_{s})\, \mathrm dM_s^{L,c} \right]
\\
& +\EE_{x,m}\!\left[ \int_{[0, {t\wedge \tau_0\wedge T_n }]}\mathrm e^{-rs} F(M^L_{s})\, \mathrm dL_s^c \right]
\\
& -  \EE_{x,m}\! \left[ \sum_{0 \le s \le t\wedge \tau_0 \wedge T_n} \mathrm e^{-rs} [w(X^L_s, M^L_s) - w(X^L_{s-} , M^L_{s-})] \right] \hskip -1mm.
\end{align*}
Let us focus on the last term on the right-hand side of this formula. We have
\begin{align*}
w(X^L_s, M^L_s) - w(X^L_{s-} , M^L_{s-})=[w(X^L_s, M^L_s) - w(X^L_{s-} , M^L_{s-})](\ind_{M^L_s=M^L_{s-}}+\ind_{M^L_s<M^L_{s-}}).
\end{align*}
Using \eqref{HJB} and \eqref{operatorT-2}, we obtain
\begin{align}
[w(X^L_s, M^L_s) - w(X^L_{s-} , M^L_{s-})]\ind_{M^L_s=M^L_{s-}}&= [w(X^L_s, M^L_{s-}) - w(X^L_{s-} , M^L_{s-})]\ind_{M^L_s=M^L_{s-}}  \notag
\\
&\le -\,(X^L_{s-}-X^L_s)F(M^L_{s-})\ind_{M^L_s=M^L_{s-}}. \label{ineq1}
\end{align}
On the other hand, $X^L_s=M^L_s$ on $\{M^L_s<M^L_{s-}\}$ and thus, using \eqref{diagonal2} and \eqref{operatorT-2}, we obtain
\begin{align}
& [w(X^L_s, M^L_s) - w(X^L_{s-} , M^L_{s-})]\ind_{M^L_s<M^L_{s-}} \notag
\\
&=[w(M^L_s, M^L_s) -w(X^L_{s-} , M^L_{s-})]\ind_{M^L_s<M^L_{s-}} \notag
\\
&=[w(M^L_s, M^L_s) - w(M^L_{s-} , M^L_{s-})+w(M^L_{s-}, M^L_{s-}) - w(X^L_{s-} , M^L_{s-})]\ind_{M^L_s<M^L_{s-}} \notag
\\
&\le -\, \{(M^L_{s-}-M^L_{s})F(M^L_{s-}) +U(M^L_s)[F(M^L_{s-})-F(M^L_{s})] + (X^L_{s-} -M^L_{s-})F(M^L_{s-})\}\ind_{M^L_s<M^L_{s-}}\notag
\\
&= -\, \{(X^L_{s-}-X^L_{s})F(M^L_{s-}) +U(M^L_s)[F(M^L_{s-})-F(M^L_{s})] \}\ind_{M^L_s<M^L_{s-}}  . \label{ineq2}
\end{align}
Summing \eqref{ineq1}--\eqref{ineq2}, substituting in \eqref{ineq0}, and observing that $\Delta L_s=-\Delta X^L_s$, we obtain
\begin{align*}
w(x,m) \ge& -\EE_{x,m} \! \left[\int_{[0, {t\wedge \tau_0\wedge T_n }]} \mathrm e^{-rs} U(M^L_{s})f(M^L_{s})\, \mathrm dM^{L,c}_s \right] \allowdisplaybreaks
\\
&+ \EE_{x,m} \! \left[\int_{[0, {t \wedge \tau_0\wedge T_n}]} \mathrm e^{-rs} F(M^L_{s-}) \, \mathrm dL_s \right]
\\
&+\EE_{x,m} \! \left[ \sum_{0 \le s \le t\wedge \tau_0 \wedge T_n} \mathrm e^{-rs}U(M^L_s) [F(M^L_{s-})-F(M^L_{s})] \right]
\end{align*}
and thus, letting $n \to \infty$ and $t \to \infty$,
\begin{align*}
w(x,m) &\geq  \EE_{x,m} \! \left[ \int_{[0,\tau_0]} \mathrm e^{-rs} F(M^L_{s-}) \, \mathrm dL_s \right] \! -\EE_{x,m} \! \left[\int_{[0,\tau_0]} \mathrm e^{-rs}  U(M^L_{s}) F(M^L_{s}) \, \mathrm dM^{L,c}_s\right]
\\
& \quad +  \EE_{x,m} \!\left[\sum_{0 \leq s \leq \tau_0} \mathrm e^{-rs}  U(M^L_{s})[F(M^L_{s-}) - F({M^L_s})] \right]
\end{align*}
for all $L \in \mathcal A(x)$ by the monotone convergence theorem. The result follows. \hfill $\blacksquare$

\section{Optimal Resource Extraction}

In this section, we provide a complete solution to the two-dimensional singular control problem (\ref{markovv}). This solution is explicit given a free-boundary function $b$ that is characterized as the solution to an ode. We proceed as follows. In Section \ref{Assumptions}, we lay down the additional assumptions under which we solve (\ref{markovv}). In Section \ref{kt}, we study an auxiliary control problem in which the dm cannot extract below a known level $m\geq 0$ of the reserves at which he receives a terminal payoff $U(m)$; this problem, which was solved in \cite{SGL} for the case $m=0$, will prove useful to set up a free-boundary problem whose solution gives a candidate value function for (\ref{markovv}). In Section \ref{gva}, we derive this free-boundary problem and characterize the associated free-boundary function $b$. In Section \ref{solution}, we fully solve (\ref{markovv}).

\subsection{Assumptions} \label{Assumptions}

We shall work under Assumptions \ref{coefficient}--\ref{density} and the following ones.

\begin{Ass}\label{ass:A2}
The functions $\mu$ and $\sigma$ are $\mathcal C^1,$ with Lipschitz derivatives$,$ and satisfy
\begin{align*}
\mu(0)>rU(0) \text{ and } \sup_{x\geq 0} \, \mu'(x) < r .
\end{align*}
\end{Ass}

\begin{Ass}\label{ass:A6}
The function $U$ is nondecreasing$,$ concave$,$ $\mathcal C^2,$ with $U'(m) \geq 1$ for $m \leq u^* $ and $U'(m) =1$ for $m \geq u^*$ for some  $u^* \geq 0,$ and satisfies
\begin{align} \label{tard}
(\CL U -r) (m) > 0 \text{ for all } m \in [0, u^*).
\end{align}
\end{Ass}

\begin{Ass}\label{ass:A12}
The function $\mu$ is twice differentiable and concave on $[0, \infty),$ and the reverse hazard rate $\frac{f}{ F}$ of $Y$ is decreasing on its support $[0,\overline y]$.
\end{Ass}

Assumption \ref{ass:A2} is of the type used in \cite{SGL}  and allows us to explicitly solve the auxiliary problem (\ref{aux}). Assumption \ref{ass:A6} sets the features of  the continuation value function $U$, which satisfies the properties of the value function of a standard extraction problem modeled as a singular control problem, as, for example, in \cite{SGL}. In particular, the threshold $u^*$ denotes the optimal extraction threshold associated with the continuation extraction problem that the dm faces once the tipping point is crossed. Inequality (\ref{tard}) implies that, the later the tipping point is crossed on the interval $[0, u^*]$, the better off is the dm. We will accordingly say that a function $U$ satisfying Assumption \ref{ass:A6} is the value function of a \textit{downgraded} extraction problem, as is the case in Example \ref{example} below. Assumption \ref{ass:A12} states that the rate at which the reserves accumulates on average decreases with the level of reserves, and that the likelihood of reaching a tipping point, assuming this has not yet occurred, increases as the level of reserves decreases. These assumptions are standard in tipping-point models (\cite{Liski, Tsur1, Tsur2}).

\begin{lemma}\label{lem:existence_barx}
Under Assumptions {\rm \ref{ass:A2}--\ref{ass:A6}}$,$ there exists $\overline x >0$ such that $\mu(x) - r U(x) >0$  on $ [0, \overline x)$ and  $\mu(x) - r U( x)<0$ on $(\overline x,  \infty),$ $\overline x \geq u^*,$ and \eqref{tard} is also satisfied on $[u^*, \overline x)$ in case $\overline x > u^*$.
\end{lemma}

\noindent \textbf{Proof.} The mapping $x \mapsto \mu(x)-rU(x)$ is differentiable and decreasing because $U'\geq 1$ and $\sup_{x\geq 0} \mu'(x)= : \overline \mu <r$. Given that $\mu(0)-rU(0)>0$, the existence and uniqueness of $\overline x$ follows from the fact that, for $x \geq u^*$ large enough,
\begin{align*}
\mu(x)-rU(x) = \mu(x)-r[x-u^*+U(u^*)] \le \mu(0) + (\overline \mu-r)x -r[U(u^*)-u^*] <0
\end{align*}
as $\overline \mu <r$. Next, as $U$ is $\CC^2$ and satisfies $U'(u^*)=1$ and $U''(u^*)=0$, we have
\begin{align*}
\mu(u^*)-rU(u^*) = \mu(u^*)U'(u^*)+ \frac{\sigma^2(u^*)}{2}\,U''(u^*) - rU(u^*) = (\CL-r)U(u^*)  \geq 0,
\end{align*}
which implies $u^* \leq \overline x$ by definition of $\overline x$. Finally, if $x \in [u^*,\overline x)$, then $U'(x)=1$ and $U''(x)=0$ and thus
\begin{align*}
(\CL-r)U(x)=\mu(x)U'(x)+ \frac{\sigma^2(x)}{2}\,U''(x) - rU(x)= \mu(x)-rU(x) >0.
\end{align*}
The result follows. \hfill $\blacksquare$

\bigskip

The fact that inequality (\ref{tard}) is satisfied on $[0, \overline x)$ implies that, the later the tipping point is crossed on the interval $[0,\overline x]$, the better off is the dm. Specifically, for each $x < \overline x$ and for any stopping time $\tau$ such that the stopped process $(X_t)_{t \leq \tau}$ stays in $[0,\overline x]$, it follows from It\^{o}'s formula that $U(x)  \leq \EE_x \hskip 0.3mm [\mathrm e^{-r\tau } U(X_\tau)]$, with a strict inequality if $\PP_x \hskip 0.3mm [\tau>0]>0$.

The following is a key implication of the inequality $U' \geq 1$ in Assumption \ref{ass:A6}.

\begin{lemma}\label{pasdesaut}
Under Assumption \ref{ass:A6}$,$ if $w\in \mathcal R(\CJ)$ is a nonnegative solution to the HJB equation \eqref{HJB} with boundary condition \eqref{diagonal1}$,$ then $w$ also satisfies the boundary condition \eqref{diagonal2}.
\end{lemma}

\noindent \textbf{Proof.} For all $m >0$ and $h\in (0,m]$, we have
\begin{align*}
&w(m,m) - w(m-h,m-h)
\\
&=\int_{m-h}^m  [w_m(s,s) + w_x(s,s)] \, \mathrm ds
\\
& \geq   \int_{m-h}^m [ U(s) f(s) + F(s) ]\, \mathrm ds
\\
& =  U(m-h) [F(m) - F(m-h)] + \int_{m-h}^m \{[U(s)-U(m-h)] f(s)+  F(s)\} \, \mathrm ds
\\
& \geq  U(m-h) [F(m) - F(m-h)] +\int_{m-h}^m [sf(s)+F(s)]\, \mathrm ds -(m-h)[F(m)-F(m-h)]
\\
& =  U(m-h) [F(m) - F(m-h)] + hF(m),
\end{align*}
where the first inequality follows from \eqref{HJB}--\eqref{diagonal1}, the second inequality follows from $U' \geq 1$, and the last equality follows from $[sF(s)]'= sf(s) + F(s)$. The result follows. \hfill $\blacksquare$

\begin{example}
\label{example} Let $U$ be the value function of an optimal extraction problem as in {\rm \cite{SGL}}$,$ whose uncontrolled reserve process has the same volatility $\sigma$ as $X^0$ and a drift $\underline \mu$ such that $\underline \mu < \mu$ on $[0,\infty),$ and with terminal payoff $C$. Assume that an optimal control corresponding to an extraction threshold $u^*\geq 0$ exists$,$ that is$,$ {\rm \cite[Theorem 4.3]{SGL}} applies. Then $U$ is nondecreasing$,$ concave$,$ $\mathcal C^2,$ and satisfies $U(0)=C,$ $U'(m)\geq 1$ for $m\leq u^*,$ and $U'(m)=1$ for $m\geq u^*$. Moreover$,$ we have
\begin{align*}
\underline \mu(x)U'(x)+ \frac{\sigma^2(x)}{2}\,U''(x) - rU(x)=0 \text{ for all } x\in [0,u^*],
\end{align*}
which implies $(\CL-r)U > 0$ on $[0,u^*]$ as $U' \geq 1$ and $\underline \mu < \mu$. Thus Assumption {\rm \ref{ass:A6}} is satisfied.
\end{example}

\subsection{An Auxiliary Problem} \label{kt}

As a useful preliminary to the remainder of this section, consider the optimal-extraction problem when the dm cannot extract below a fixed and known level $m$ of the reserves, at which he receives a terminal payoff $U(m)$:
\begin{align} \label{aux}
V^m(x) := \sup_{L \in {\cal A}^m(x)} \EE_x \! \left [\int_{[0,\tau_{m}]} \mathrm e^{-rs} \, \mathrm dL_s + \mathrm e^{-r\tau_m}U(m) \right ] \hskip -1mm, \quad x \geq m,
\end{align}
where
\begin{align*}
\CA^m (x): = \big\{L \in \LL :  (X_{t-}^L-m)^+-(\Delta L)_t \ge 0 \text{ for all } t \ge 0 \big\}.
\end{align*}
Our next result, the proof of which is provided in the Appendix for the sake of completeness, is a direct consequence of \cite[Theorem 4.3]{SGL}.

\begin{proposition} \label{bench}
The following holds$:$
\begin{itemize}

\item[(i)]

If $m\geq \overline x,$ then the value function $V^m$ of problem \eqref{aux} is given by
\begin{align}
V^m(x) = x - m + U(m) \text{ for all } x\geq m, \label{benchover}
\end{align}
and the policy of initially extracting $x -m$ is optimal.

\item[(ii)]

If $m< \overline x,$ then the value function $V^m$ of problem \eqref{aux} is concave$,$ $\CC ^2,$ and satisfies $V^m(m) = U(m)$ and the HJB equation
\begin{align} \label{HJBbench}
\max\big\{(\CL-r)V^m(x), 1 - V^m_x(x) \big\} =0 \text{ for all } x\geq m.
\end{align}
Moreover$,$
\begin{align}\label{eq:formulaVm}
V^m (x)= \begin{cases}
A^m \psi(x) + B^m \phi(x) & \text{if }\, m\leq x \leq \eta(m)
\\
x - \eta(m) + \frac{\mu(\eta(m))}{r} &\text{if }\,  x \geq \eta(m)
\end{cases}
\end{align}
for $\psi$ and $\phi$ two positive fundamental solutions to the ode $\mathcal L u - ru=0,$ respectively increasing and decreasing$,$ and
\begin{align}
A^m := \frac{\phi(m) - U(m) \phi'(\eta(m))}{D(\eta(m),m)}, \label{Aaux}
\\
B^m := \frac{U(m) \psi'(\eta(m)) - \psi(m)}{D(\eta(m), m)}, \label{Baux}
\end{align}
with
\begin{align}
D(x,m): = \psi'(x) \phi(m) - \phi'(x) \psi(m) >0 \label{aux1}
\end{align}
for all $x \geq m \geq 0,$ and $\eta(m) \in (m,\overline x)$ defined as the unique solution in $x$ to
\begin{align} \label{aux2}
N(x,m) : =\phi(x)\psi(m)-\psi(x)\phi(m)+ \frac{\mu(x)}{r}\,D(x,m) - D(x,x)U(m) = 0.
\end{align}

\item[(iii)]

The process $L^m \in \CA^m(x)$ such that $(\Delta L^m)_0 = [x-\eta(m)]^+$ and $(X^{L^m},L^m)$ is the unique solution to
\begin{align} \label{eq:def_Lm}
\left\{ \begin{array} {rcl}
\mathrm dX^{L^m}_t \hskip -2.5mm  &=& \hskip -2mm \mu(X^{L^m}_t)\, \mathrm dt+ \sigma(X^{L^m}_t)\, \mathrm dB_t-\mathrm dL^m_t, \quad X^{L^m}_ {0-}=x
\\
L^m_t \hskip -2mm & =& \hskip -2mm \sup_{0\leq s \leq t} \big[x+ \int_0^s \mu(X^{L^m}_u) \,\mathrm du + \int_0^s \sigma(X^{L^m}_u) \,\mathrm dB_u - \eta(m)\big]^+
\end{array}
\right.
\end{align}
is an optimal extraction process.

\item[(iv)]

Letting $x^0:=\eta(0),$ the mapping $\eta: [0, \overline x] \to [x^0, \overline x]:m \mapsto \eta(m)$ is increasing$,$ $\CC^1,$ and satisfies $\eta(m) > m$ for all $m \in [0,\overline x)$ and $\eta(\overline x) = \overline x$. Moreover$,$ for each $m\in [0, \overline x],$ $N(x,m)>0$ for all $x <\eta(m)$ and $N(x,m)< 0$ for all $x > \eta(m)$.
\end{itemize}
\end{proposition}

Claim (i) states \pagebreak that, if $m$ is larger than the threshold $\overline x$ defined in Lemma \ref{lem:existence_barx}, then a lump-sum extraction is optimal; that is, the dm is better off extracting immediately the reserves all the way down to $m$ in order to get the terminal payment $U(m)$ right now. Claim (ii) states that, if $m$ is below the threshold $\overline x$, then the dm optimally allows the reserves to grow up to the extraction threshold $\eta(m)$ and extracts any reserve above $\eta(m)$, causing the optimally controlled reserve process $X^{L^m}$ to be reflected downwards at $\eta(m)$. Claim (iii) provides useful and intuitive properties of the extracting threshold $\eta(m)$; intuitively, the higher $m \in [0, \overline x]$ is, the higher the reserve requirement $\eta(m)$ before extraction becomes optimal, and the equality $\eta (\overline x) = \overline x$ reflects the fact that, if the initial level of reserves is larger than $\overline x$, then the dm optimally extracts the surplus $x - \overline x >0$.

\subsection{A Free-Boundary Problem} \label{gva}

To solve the singular control problem \eqref{markovv}, we use a standard guess-and-verify approach. We guess that, when the reserves are above a threshold that depends on its running minimum, it is optimal to extract. Furthermore, we guess from Proposition \ref{bench} that, if $m$ is sufficiently large, say, larger than a threshold $\overline m$, it is optimal to extract immediately the surplus $x - \overline m >0$. Therefore, we conjecture the
existence of a free-boundary function $b:[0, \overline m] \mapsto [0, \overline m]$, with $b(m) \geq m$, separating the inaction region
\begin{align*}
\{(x,m) \in \CJ:  m \leq \overline m \text{ and } m \leq x<b(m)\}
\end{align*}
from the extraction region
\begin{align*}
\{(x,m)\in \CJ : m> \overline m \text{ or } ( m\in [0, \overline m] \text{ and } x \ge b(m) )\}.
\end{align*}
where the control process $L$ is activated. We also guess that $b(m) \leq \eta (m)$ for all $m \in [0, \overline m]$: intuitively, uncertainty on the tipping point $Y \leq m$ should make the dm less cautious and increase his value compared to the case where the dm knows for certainty that extraction will end with certainty when the reserves hit the level $m$. We also expect that $b$, like $\eta$, is nondecreasing, satisfies $b(m) \geq m$ for all $m \in [ 0, \overline m]$, with $b(0) = x^0$ and $b (\overline m) = \overline m \leq \overline x$. Finally, because the dm learns nothing about the advent of a tipping point when the running minimum $m$ is larger than $\overline y$, the upper bound of the support of $Y$, it is natural to conjecture that the function $b$ is constant on $[\overline y, \infty)$.

Relying on the verification Lemma \ref{verif}, we look for a solution to the HJB equation \eqref{HJB} satisfying the boundary conditions \eqref{diagonal1}--\eqref{diagonal2}. More precisely, we aim at finding a pair of functions $(W,b)$ with $W\in \mathcal{R}(\CJ)$ and a threshold $\overline{m} > x^0$ that solve the following free-boundary problem:
\begin{align}
b(0) & =  x^0, \label{c1}
\\
b(\overline m) & =  \overline m, \label{m} \allowdisplaybreaks
\\
{\cal L}W(x,m) - r W(x,m) &= 0,  \quad 0< m \le \overline{m}  \mbox{ and } 0< m <x <b(m), \label{vs1}
\\
W_x (x, m) &=  F(m),  \quad  0 < m \leq \overline m   \mbox{ and } x \geq b(m)  \label{vs2}
\\
 W_m (m,m) & = U(m)f(m), \quad 0 < m \leq \overline m, \label{vs2'}
 \\
 W_m (m,m) &\ge U(m)f(m), \quad m > \overline{m},  \label{vs2''}
 \\
 W(m,m)&\ge T[W](m), \quad m >0,  \label{vs2''''}
 \\
W_{xx}(b(m),m)&= 0, \quad 0 < m\leq \overline m . \label{vs3}
\end{align}
We show below that, if a solution to (\ref{c1})--(\ref{vs3}) exists, then the function $b$ describing the free boundary satisfies an ode. We first observe that a solution to \eqref{vs1} writes as
\begin{align*}
W(x,m) = A(m) \psi ( x) + B(m) \phi(x),
\end{align*}
for some functions $A$ and $B$. Then, (\ref{vs2}) and (\ref{vs3}) yield
\begin{align}
A(m) & =  \frac{F(m)\phi^{\prime\prime} (b(m))}{\psi'(b(m)) \phi^{\prime\prime}(b(m)) - \phi'(b(m))\psi^{\prime\prime}(b(m))}, \label{A}
\\
B(m) &=  \frac{- F(m)\psi^{\prime\prime} (b(m))}{\psi'(b(m)) \phi^{\prime\prime}(b(m)) - \phi'(b(m))\psi^{\prime\prime}(b(m)) }. \label{B}
\end{align}
Notice that $\phi$ and $\psi$ satisfy
\begin{align}\label{eq:ODE_phi_psi_1}
u''  =\frac{2}{\sigma^2} \,(r u  - \mu  u' )
\end{align}
on $[0, \infty)$. Letting $D(x) : = \psi'(x) \phi(x) - \phi'(x) \psi(x)$, a direct computation leads to
\begin{align}
A(m) & =  \frac{F(m)}{D(b(m))}\left [\phi (b(m)) - \frac{\mu(b(m))}{r}\,\phi'(b(m))\right ] \hskip -1mm, \label{A1}
\\
B(m)  &=  \frac{-  F(m)}{D(b(m))} \left [\psi(b(m)) - \frac{\mu(b(m))}{r} \, \psi'(b(m)) \right ] \hskip -1mm. \label{B1}
\end{align}
Taking the derivatives with respect to $m\in (0, \overline m \wedge \overline y)$ yields
\begin{align*}
A'(m) & =  \frac{b'(m)F(m)}{D(b(m))}  \left\{ \phi'(b(m)) - \frac{\mu(b(m))}{r} \,\phi^{\prime\prime}(b(m)) - \frac{\mu'(b(m))}{r} \,\phi'(b(m)) \right.
\\
& \left. \quad \quad \quad \quad\quad\quad\quad -\, \frac{D'(b(m))}{D(b(m))}\left[\phi(b(m)) - \frac{\mu(b(m))}{r} \, \phi'(b(m))\right]\right\}
\\
 &  \quad  +  \frac{f(m)}{D(b(m))} \left[\phi(b(m)) - \frac{\mu(b(m))}{r}\,\phi'(b(m))\right] \hskip -1mm,
 \\
B'(m) & =  \frac{- b'(m)F(m)}{D(b(m))} \left\{ \psi'(b(m)) - \frac{\mu(b(m))}{r} \,\psi^{\prime\prime}(b(m)) - \frac{\mu'(b(m))}{r} \,\psi'(b(m)) \right.
\\
& \left. \quad \quad \quad \quad\quad\quad\;\;\,\quad -\, \frac{D'(b(m))}{D(b(m))}\left[\psi(b(m)) - \frac{\mu(b(m))}{r} \, \psi'(b(m))\right]\right\}
\\
 &   \quad  + \frac{f(m)}{D(b(m))} \left[\psi(b(m)) - \frac{\mu(b(m))}{r}\,\psi'(b(m))\right] \hskip -1mm.
\end{align*}
Using again \eqref{eq:ODE_phi_psi_1} and the resulting relation $\frac{D'(b(m))}{D(b(m))}= - \frac{2 \mu(b(m))}{\sigma^2(b(m))}$, we obtain
\begin{align*}
&A'(m)
\\
&= \frac{1}{D(b(m))} \left\{b'(m) F(m) \!\left[1 - \frac{\mu'(b(m))}{r}\right]\! \phi'(b(m)) + f(m) \! \left[\phi(b(m)) - \frac{\mu(b(m))}{r} \,\phi'(b(m))\right] \right\} \hskip -0.5mm,
\\
&B'(m)
\\
& =  \frac{-1}{D(b(m))} \left\{b'(m) F(m) \!\left[1 - \frac{\mu'(b(m))}{r}\right]\! \psi'(b(m)) + f(m) \! \left[\psi(b(m)) - \frac{\mu(b(m))}{r} \,\psi'(b(m))\right] \right\} \hskip -0.5mm.
\end{align*}
Substituting these formulas into \eqref{vs2'} yields the following ode:
\begin{align} \label{ode1}
b'(m)  = E(b(m),m), \quad  m \in (0,\overline m\wedge \overline y),
\end{align}
where the vector field $E$ on $(0, \infty)^2$ is defined by
\begin{align} \label{vecfield}
E(x,m) : =\frac{ H(m) N(x,m)}{G(x) D(x,m)},
\end{align}
where $D(x,m)$ and $N(x,m)$ are given by \eqref{aux1}--\eqref{aux2}, $H(m) := \frac{f(m)}{F(m)}$, and $G(x) := 1 - \frac{\mu'(x)}{r}$. Extending artificially $b$ to $[0,\overline x]$ and taking into account that $b$ is constant beyond $\overline y$, we are thus led to the Cauchy problem of finding a function $b: [0,\overline x] \rightarrow [0,\overline x]$ such that
\begin{align}\label{ode2}
\left\{ \begin{array} {rcll}
b \hskip -2mm &\in&  \hskip -2mm \CC([0,\overline x]) \cap \CC^1((0,\overline x \wedge \overline y))
\\
b(0) \hskip -2mm  &=&\hskip -2mm  x^0
\\
b'(m)\hskip -2mm  &=&\hskip -2mm  E(b(m),m), \quad m \in(0,\overline x \wedge \overline y)
\\
b(m) \hskip -2mm &=& \hskip -2mm  b(\overline x \wedge \overline y), \quad m \in (\overline x \wedge \overline y, \overline x \vee \overline y]
\end{array} \right.\hskip -2mm  .
\end{align}
The vector field $E$ in \eqref{vecfield} is well-defined on $(0, \infty)^2$ because $F> 0$ on $(0,\infty)$, $D >0$ on $(0, \infty)^2$, and $G >0$ on $(0, \infty)$ by Assumption \ref{ass:A2}; notice also that $E$ vanishes for $m > \overline y$. Because $H$ is locally Lipschitz on $(0,\overline y]$ by Assumption \ref{density}, $E$ is locally Lipschitz on $(0, \infty)\times (0,\overline y]$. However, $E$ cannot be continuously extended to $(0, \infty) \times [0,\overline y]$ as  $\lim_{m \to 0} H(m) = \infty$, which implies from Proposition \ref{bench}(iii) that, for each $x > 0$ with $x\neq x^0$, $\lim_{m \to 0} E(x,m) = \infty$. Thus the Cauchy problem \eqref{ode2} requires a specific analysis, to which we now turn.

\begin{proposition} \label{lemedo}
There exists a unique solution $b$ to the Cauchy problem \eqref{ode2}. This solution is nondecreasing and satisfies $b(m) < \eta(m)$ for all $m \in (0, \overline{x}]$. Moreover$,$ the point $\overline{m}:=\min \hskip 0.5mm \{m \in [0,\overline x] : b(m)=m\} $ is well-defined$,$ belongs to $(0, \overline{x}),$ and is such that $b'(\overline m) <1$.
\end{proposition}

\noindent \textbf{Proof.} The proof consists of three parts.

\subparagraph{Part I: Existence}

To avoid dealing with several cases depending on the value of $\overline y$, we consider an auxiliary well-posed ode. Let $\hat f$ denote a Lipschitz positive function on $[0,\infty)$ that coincides with $f$ on $[0,\overline y]$, and let $\hat F(m):=\int_0^m \hat f(t) \, \mathrm dt$  and $\hat H(m): =\frac{\hat f(m)}{\hat{F}(m)}$ for all $m>0$. Define then the vector field $\hat E$ on $(0,\infty)^2$ by
\begin{align*}
\hat E(x,m) :=\frac{ \hat H(m)N(x,m) }{G(x) D(x,m)},
\end{align*}
which is locally Lipschitz on $(0,\infty)^2$ and coincides with $E$ on $(0,\infty)\times (0,\overline y]$. Notice that, because $\hat H$, $D$, and $G$ are positive, the sign of $\hat E$ coincides with that of $N$, so that, for all $x>0$ and $m\in (0,\overline x]$, $\hat E(x,m)\gtreqless 0$ if $x \lesseqgtr \eta(m)$ by Proposition \ref{bench}(iv). For all $m_0 \in (0,\overline x]$ and $x_0> m_0$, we let $b_{x_0, m_0}$ be the unique maximal solution to the ode $b'(m) = \hat E(b(m), m)$ such that $b(m_0)=x_0$, whose existence follows from the Cauchy--Lipschitz theorem. The proof that there exists a solution to the Cauchy problem \eqref{ode2} now consists of two steps.

\subparagraph{Step 1}

For each $m_0 \in (0,\overline x]$, let $(x_n)_{n \geq 1}$ be a decreasing sequence such that $x_n \to \eta(m_0)$. As $\eta$ is increasing by Proposition \ref{bench}(iv) and $\hat E(x_n ,m_0)<0$ as $x_n > \eta(m_0)$, it follows from standard arguments that, for each $n \geq 1$, $b_{x_n, m_0}$ is decreasing over $(0, m_0]$. From the non-crossing property of the solutions to the ode $b'(m) = E(b(m),m)$, $(b_{x_n, m_0})_{n\geq 1}$ is a decreasing sequence of decreasing functions bounded below by $\eta(m_0)$ over $(0, m_0]$. Over this interval, it thus admits a nonincreasing pointwise limit $\underline b : (0, m_0] \to [\eta(m_0),\infty)$ such that $\underline b(\eta(m_0)) =  \eta(m_0)$. For all $n\geq 1$ and $m\in( 0, m_0]$, we have
\begin{align*}
b_{x_n, m_0 } (m) =  x_n - \int_m^{m_0} \hat E(b_{x_n, m_0 }(s),s) \, \mathrm ds.
\end{align*}
Because $b_{x_n, m_0 }(s) \in [\eta(m_0), b_{x_1, m_0 }(m)]$ for all $n\geq 1$ and $s \in [m,m_0]$ and $E$ is continuous over the compact domain $[\eta(m_0), b_{x_1, m_0 }(m)] \times [m , m_0]$, it follows by bounded convergence that
\begin{align*}
\underline b (m) =  \eta(m_0) - \int_m^{m_0} \hat E(\underline b (s),s) \, \mathrm ds.
\end{align*}
Because this is true for all $m \in (0,m_0]$, we deduce that $\underline b$ is the restriction to $(0, m_0]$ of $b_{\eta(m_0), m_0}$; in particular, $b_{\eta(m_0), m_0}$ is nonincreasing over $(0, m_0]$, and thus $b_{\eta(m_0), m_0}(m) \geq \eta(m_0)$ for all $m \in (0,m_0]$. Geometrically, as $\eta$ is increasing by Proposition \ref{bench}(iv), the upshot of this discussion is that, in the $(x,m)$-space, the integral curves of the ode $b'(m) = \hat E(b(m), m)$ can cross the curve $\{(x,m) : 0< m \leq \overline x \text{ and } x = \eta(m) \}$ at most once, from below.

\subparagraph{Step 2}

Now, let $(m_n)_{n \geq 1}$ be a decreasing sequence such that $m_1 < \eta(0)$ and $ m_ n \to 0$. Because $\eta$ is increasing and $\eta(0)=x^0$ by definition, $b_{x^0, m_n}(m_n) = x^0 < \eta(m_n)$ for all $n \geq 1$. It thus follows from Step 1 that, for each $n \geq 1$, $b_{x^0,m_n}$ is well-defined on a neighborhood of $[m_n, \overline x]$ and is increasing and strictly bounded from above by $\eta$ on $[m_n,\overline x)$. Let then
\begin{align*}
b_n(m) :=  \begin{cases} x^0 &  \text{if } \, m \in [0, m_n) \\ b_{ x^0, m_n}(m)  & \text{if } \, m \in [m_n, \overline x] \end{cases}, \quad n \geq 1, \, m \in [0, \overline x].
\end{align*}
By construction, $(b_n)_{n\geq 1}$ is a nondecreasing sequence of positive nondecreasing functions on $[0, \overline x]$ bounded above by $\eta$. It thus admits a pointwise limit $\overline b$ such that $\overline b(0) =  x^0$. For all $n\geq 1$ and $m \in [0, \overline x]$, we have
\begin{align*}
b_n(m) =  x^0 + \int_0^m \hat E(b_n(s),s) \ind_{s \geq m_n} \, \mathrm ds.
\end{align*}
In particular, for $0 < m' < m < \overline x$,
\begin{align*}
b_n(m) - b_n(m') = \int_{m'}^m \hat E(b_n(s),s) \ind_{s \geq m_n} \, \mathrm ds.
\end{align*}
Because $b_n(s) \in [x^0, \eta(s)]$ for all $n\geq 1$ and $s \in [m',m]$ and $E$ is continuous over the compact domain $\{(x,s) : s \in [m',m] \text{ and } x \in [x^0, \eta(s)]\}$, it follows by bounded convergence that
\begin{align}
\overline b(m) - \overline b(m') = \int_{m'}^m \hat   E( \overline b(s), s)  \, \mathrm ds, \label{bode'}
\end{align}
Because $\overline b $ is bounded above by $\eta$, we have $\overline b(0+)\le \eta(0)=x^0$, so that $\overline b$ is continuous at $0$. Letting $m' \to 0$ in (\ref{bode'}), it follows by monotone convergence that
\begin{align} \label{soledo}
\overline b(m) -  \overline b(0) = \int_0^m \hat E(\overline b(s), s) \, \mathrm ds <\infty,
\end{align}
which shows that $\overline b$ satisfies the ode $b(m)=\hat E(b(m),m)$ on $(0,\overline x)$. Thus the function $b:[0,\overline x] \to [0,\overline x]$ defined by
\begin{align*}
b(m) : =\begin{cases}
\overline b(m) & \text{if }\,  m \in [0,\overline x\wedge \overline y]
\\
\overline b(\overline x\wedge \overline y) & \text{if }\, m \in (\overline x \wedge \overline y, \overline x \vee \overline y]
\end{cases}
\end{align*}
is a solution to the Cauchy problem \eqref{ode2}.

\subparagraph{Part II: Uniqueness}

Suppose, by way of contradiction, that $b$ and $\tilde b$ are different solutions to \eqref{ode2}. Because $E$ is locally Lipschitz on $(0, \infty)\times (0,\overline y]$, this implies that there exists $m >0$ such that, say, $\tilde b (s)> b (s)$ for all $s \in (0,m]$. With no loss of generality, suppose that $m\in (0, \overline y \wedge \overline x)$ is sufficiently small that $b(m)>m$. From \eqref{soledo}, we have
\begin{align} \label{unique}
\tilde b(m) -  b(m) = \int_0^m [\hat E(\tilde b(s), s) - \hat E(b(s),s)]\, \mathrm ds.
\end{align}
The left-hand side of \eqref{unique} is positive. However, by Lemma \ref{forodegenegood}(ii) in Appendix \ref{AuxLemmas}, the mapping $x\mapsto \hat E(x,s)$ is decreasing on $[s,\eta(s)]$ for all $s \in (0,m]$. As $\eta(s) > \tilde b(s) >b(s) >s$ for all $s \in (0,m]$, it follows that the right-hand side of \eqref{unique} is negative, a contradiction. Thus the solution to the Cauchy problem \eqref{ode2} is unique.

\subparagraph{Part III: Properties of $\boldsymbol{\overline m}$}

By the intermediate value theorem, $\overline m$ is well-defined as $b(0)=x^0>0$ and $b(\overline x)< \eta(\overline x)=\overline x$ by Step 1; by continuity, it must be that $b(\overline m) = \overline m < \overline x$. To complete the proof, we must only show that $b'(\overline m) <1$. We distinguish two cases.

\subparagraph{Case 1}

If $\overline m >\overline y$, then it must be that $\overline y <\overline x$ and thus $\overline m = b(\overline y)\leq \eta (\overline y)<\overline x$ and $b'(\overline m) =0<1$, as desired.

\subparagraph{Case 2}

If $\overline m \leq \overline y$, \pagebreak then, because $b(m)\geq m$, it must be that $b'(\overline m)  \leq 1$ by definition of $\overline m$. Suppose, by way of contradiction, that $b'(\overline m) =1$, and, for $\tilde m \in (0, \overline m)$ and $\tilde x \in (\tilde m, b(\tilde m))$, consider $b_{\tilde x, \tilde m}$. Because $b_{\tilde x, \tilde m} (\tilde m) < b (\tilde m)$, we have $b_{\tilde x, \tilde m} < b$ on $(\tilde m, \overline m)$. Moreover, for each $m \in  (\tilde m, \overline m)$, it must be that $b_{\tilde x, \tilde m} (m)> m$. Indeed, otherwise, letting $\hat m :=  \inf \hskip 0.5mm \{m\in (\tilde m, \overline m):b_{\tilde x, \tilde m} (m) = m\}$, we would have $b_{\tilde x, \tilde m} (\hat m) = \hat m$ and $E(\hat m, \hat m) = b_{\tilde x, \tilde m}' (\hat m) \leq 1 = b'(\overline m) = E(\overline m, \overline m)$, in contradiction with Lemma \ref{forodegenegood}(i) in Appendix \ref{AuxLemmas} along with the fact that $H$ is decreasing over $[\tilde m, \overline m] \subset (0, \overline y]$ by Assumption \ref{ass:A12}. Thus $m < b_{\tilde x, \tilde m} (m) < b(m)$ for all $m \in (\tilde m, \overline m)$, from which it follows by continuity that $b_{\tilde x, \tilde m} (\overline m) =b(\overline m)$ as $\overline m =b(\overline m)$. This, by the Cauchy--Lipschitz theorem, however implies $b_{ \tilde x, \tilde m} = b$, in contradiction with $b_{ \tilde x, \tilde m} (\tilde m) < b(\tilde m)$. Thus $b'(\overline m)  < 1$, as desired. Hence the result. \hfill $\blacksquare$

\subsection{Solution to the Tipping-Point Problem} \label{solution}

In this section, we characterize the value function of the singular control problem \eqref{markovv} and the associated optimal control when it exists. The previous section suggests three key properties: first, the value function of \eqref{markovv} satisfies the free-boundary problem (\ref{c1})--(\ref{vs3}); second, for any state $(x,m)\in \mathcal J$ such that $0 < m <\overline m$ and $x \geq b(m)$, it is optimal to immediately extract $x - b(m)$; third, for any state $(x,m)\in \mathcal J$ such that $\overline m \leq m < x$, it is optimal to immediately extract $x-m$. However, it is unclear whether it is optimal to immediately extract $m- \overline m$ in a state $(m,m)\in \mathcal J$ such that $m > \overline m $.  We will see that the answer depends on the relative positions of $\overline m$ and of the optimal extraction threshold $u^*$ for the downgraded problem.

\subsubsection{The Candidate Value Function}

Our next result derives the candidate value function for our tipping-point problem.

\begin{proposition} \label{verif'}
Given the solution $b$ to the Cauchy problem \eqref{ode2} and the functions $A$ and $B$ defined in \eqref{A}--\eqref{B}$,$ the function $W$ defined on ${\cal J}$ by
\begin{align}
W(x,m) & := A(m) \psi ( x) + B(m) \phi(x), \quad 0 \leq m \leq\overline m  \text{ and }  m \leq x \leq b(m), \label{d1}
\\
W(x,m)&:= [x - b(m)] F(m) +\frac{\mu (b (m))}{r}\, F(m),   \quad 0 \leq m \leq \overline m  \text{ and } x>b(m) , \label{d11}
\\
W(x,m) &:= \frac{\mu(\overline m)}{r} \,F(\overline m)+ ( x - m) F(m) +\int_{\overline m}^m [U(s)f(s)+F(s)] \, \mathrm ds ,\quad  x \geq m > \overline m, \label{d12}
\end{align}
satisfies $W\geq V,$ where $V$ is the value function of the singular control problem \eqref{markovv}.
\end{proposition}

\noindent \textbf{Proof.} Consider the three domains
\begin{align*}
{\cal J}_1 &:= \{(x,m) \in \CJ :   m \leq\overline m \text{ and } m \le x \leq b(m) \},
\\
{\cal J}_2 &:=\{(x,m) \in \CJ : m \leq \overline m \text{ and } x \geq b(m) \},
\\
{\cal J}_3 &: =\{(x,m) \in {\mathcal J}:  m \geq \overline m\}.
\end{align*}
Lemma \ref{regulW} in Appendix \ref{AuxLemmas} shows that $W\in \mathcal R(\CJ)$ and in particular is continuous on $\mathcal{J}$. By Lemmas \ref{verif} and \ref{pasdesaut}, it is sufficient to prove that $W$ defined by (\ref{d1})--(\ref{d12}) is a solution to the HJB equation \eqref{HJB} with boundary condition \eqref{diagonal1} on the three subdomains $\CJ_1$, $\CJ_2$, and $\CJ_3$.

\medskip

(1) On $\CJ_1$, given the solution $b$ to the Cauchy problem \eqref{ode2} and the functions $A$ and $B$ defined by (\ref{A})--(\ref{B}), and recalling that $b$ is constant on $[\overline y \wedge \overline m, \overline m]$, direct computations mimicking those detailed in Section \ref{gva} imply that $\CL W-rW=0$ and that the boundary conditions \eqref{vs2}--\eqref{vs2'} are satisfied; in particular, $W$ satisfies \eqref{diagonal1}. Finally, it follows from $W_x((b(m),m) = F(m)$ and $W_{xx}((b(m),m) = 0$ that the mapping $x \mapsto W(x,m)$ is concave on $[m,b(m)]$ for all $m \in (0, \overline{m}]$ (\cite[Lemma 4.2(c)]{SGL}), which implies that $W_x(x,m) \ge F(m)$ for all $(x,m) \in \CJ_1$. Thus $W$ satisfies \eqref{HJB}.

\medskip

(2) On $\CJ_2$, $W(x,m)$ is affine in $x$ with slope $F(m)$. Moreover, for each $(x,m) \in {\cal J}_2$,
\begin{align} \label{verifineq2}
{\cal L}W(x,m) -rW(x,m) = \{[\mu(x)  - rx] -   [\mu (b(m)) -r b(m)]\}F(m) \le 0,
\end{align}
where the inequality follows from Assumption \ref{ass:A2}. Therefore, $W$ satisfies \eqref{HJB}, and \eqref{diagonal1} is irrelevant as $x > b(m) \geq m$.

\medskip

(3) On $\CJ_3$, $W(x,m)$ is affine in $x$ with slope $F(m)$ and $W_m(m,m) = U(m) f(m)$, so that $W$ satisfies \eqref{diagonal1}. There only remains to show that $\CL W-rW\leq 0$. For each $(x,m) \in {\cal J}_3$,
\begin{align}
&{\cal L}W(x,m) - rW(x,m) \notag
\\
&= [\mu(x) - rx]F(m)  - \mu(\overline m) F(\overline m) + r m F(m) - r \int_{\overline m}^m [U(s) f(s) +F(s)] \, \mathrm ds,  \notag
\\
& \leq  \mu(m)F(m)  - \mu(\overline m) F(\overline m)  - r \int_{\overline m}^m [U(s) f(s) +F(s)] \, \mathrm ds \notag
\\
&: = R(m),  \label{comfy}
\end{align}
where the inequality follows from Assumption \ref{ass:A2}. Notice that $R(\overline m) =0$ and
\begin{align*}
R'(m) = r F(m)  \left \{ \frac{\mu'(m)}{r} -1 + H(m) \!\left[\frac{\mu(m)}{r} - U(m)\right]\right \}\!,
\end{align*}
where $R'(\overline y):=R'(\overline y - )$ by convention whenever $\overline y \geq \overline m$. For each $m \geq \overline m$, we have
\begin{align*}
H(m) \,{\frac{\mu( m)}{r}-U( m) \over 1-\frac{\mu'( m)}{r}} = E( m,  m) \leq E(\overline m, \overline m) < 1,
\end{align*}
where the first inequality follows from Assumption \ref{ass:A12} and Lemma \ref{forodegenegood}(i), and the second inequality follows from Proposition \ref{lemedo}. It follows that $R'(m) \leq 0$ and hence $R(m) \leq 0$ for all $m \geq \overline m$, which implies by \eqref{comfy} that $W$ satisfies \eqref{HJB}. Hence the result. \hfill $\blacksquare$

\bigskip

Some comments on our candidate value function $W$ are in order. The expression \eqref{d1} for $W$ on $\CJ_1$ follows from our analysis of \eqref{vs1} and from Proposition \ref{lemedo}, and is intuitively associated with a control that horizontally reflects  the process $X$ at the boundary $b$. The expression of $W$ on $\CJ_2$ follows from \eqref{operatorT-2}, which our candidate assumes to be an equality for $m < x' <x$ and $x' = b(m)$, and corresponds to an initial extraction of $x-b(m)$. The equality $W(b(m), m) =\frac{\mu(b(m))}{r}\,F(m)$ follows from {\eqref{vs1}}--{\eqref{vs2}} and {\eqref{vs3}}. In order to discuss the expression of $W$ on $\CJ_3$, it is useful to distinguish the two cases $u^* \leq \overline m$ and  $u^* > \overline m$.\footnote{Notice that, as the vector field (\ref{vecfield}) depends on $U$ and thus $\overline m$ depends on $u^*$, it is not obvious that these two cases can arise for different specifications of the model. To see that the case $u^* < \overline m$ must be considered, simply set $U(x) \equiv x$, so that $u^*=0 < x^0 = \eta(0) = b(0) < b(\overline m) = \overline m$. We show in Appendix \ref{Apptwocases} that the case $u^* > \overline m$ is also relevant.}

\paragraph{Case 1: $\boldsymbol{u^* \leq \overline m}$}

Intuitively, in that case, if $x > \overline m$, then the dm, upon immediately extracting $x- \overline m$, does not risk ending up in a situation in which he would have extracted too much, should he thereby have crossed the tipping point $Y$. To see this, notice that for $m >m'\geq u^*$, $U(m)= U(m')+m-m'$. Therefore, if $x \geq m > \overline m \geq u^*$, then \eqref{d12} can be rewritten as
\begin{align}
&W(x,m) \notag
\\
&=\frac{\mu(\overline m)}{r} \, F(\overline m)+ ( x - m) F(m) +\int_{\overline m}^m \{[U(\overline m)+(s-\overline m)]f(s)+F(s)\}\, \mathrm ds \notag
\\
&= \frac{\mu(\overline m)}{r} \,F(\overline m)+ ( x - m) F(m)+ [U(\overline m)-\overline m][F(m)-F(\overline m)]+\int_{\overline m}^m [sf(s)+F(s)] \, \mathrm ds \notag
\\
&= \frac{\mu(\overline m)}{r} \, F(\overline m)+ ( x - m) F(m)+ [U(\overline m)-\overline m][F(m)-F(\overline m)]+mF(m)-\overline m F(\overline m) \notag \\
&= F(m) \left\{x - \overline m + {W(\overline m,\overline m) \over F(\overline m)} \left[{F(\overline m) \over F(m)}\right] \!+ U(\overline m)\! \left[1-{F(\overline m) \over F(m)}\right]\right\}\!, \label{djump}
\end{align}
where the third equality follows from $[sF(s)]'=sf(s)+F(s)$. The expression \eqref{djump} reflects that, if the dm extracts $x-\overline m$ at time 0, two scenarios can occur: either $Y < \overline m$, with conditional probability $F(\overline m) \over F(m)$, in which case the dm obtains the continuation payoff $W(\overline m, \overline m) \over F(\overline m)$ associated with the control that reflects $X$ on the boundary $b$ starting from $(\overline m,\overline m)$, or $Y\in[\overline m, m)$, with conditional probability $1- {F(\overline m) \over F(m)}$, in which case the dm obtains the continuation payoff $U(\overline m)$. The intuition is that, because $U$ is affine over $[\overline m, m) \subset [u^*, \infty)$, the dm is risk-neutral with respect to the risk that $Y\in[\overline m, m)$, which makes it optimal for him to immediately extract $x- \overline m$ starting from $(x,m)$. The reason is that the optimal extraction threshold $u^*$ associated with the downgraded value function $U$ is less than $\overline m$, so that the optimal control in the downgraded problem, consisting in immediately extracting $z-u^*$ when starting at a level of reserves $z \geq u^*$, is consistent with the immediate extraction of $x-\overline m$ in the tipping-point problem. This course of action is thus optimal in the tipping- point problem as it causes no regret in the downgraded problem, reflecting that, because the dm has the opportunity to immediately extract a larger quantity in the latter problem even if the tipping point is reached in the interval $[\overline m, m)$, the marginal value of one additional unit of reserves in both problems is equal to 1.

\paragraph{Case 2: $\boldsymbol{u^* > \overline m}$}

In that case, it is convenient to split  ${\cal J}_3 $ into two subdomains
\begin{align*}
{\cal J}'_3 &:= \{(x,m)  \in {\mathcal J} : \overline m \le m < u^* \} \,\text{ and } \hskip 0.3mm  {\cal J}''_{3} : =\{(x,m)  \in {\mathcal J}: m \geq u^*\}.
\end{align*}
In analogy with \eqref{djump}, using that $U'=1$ on $[u^*,\infty)$, we have, for $(x,m)\in {\cal J}''_{3}$,
\begin{align*}
W(x,m) = F(m) \left\{x - u^* + {W(u^*,u^*) \over F(u^*)} \left[{F(u^*) \over F(m)}\right] + U(u^*)\! \left[1-{F(u^*) \over F(m)}\right]\right\}\!,
\end{align*}
which intuitively calls for immediately extracting $x-u^*$. To grasp the intuition of \eqref{d12} on the set ${\cal J}'_3$, take $(x,m) \in {\cal J}$ with $m \in [\overline m, u^*)$ and consider as above a control that involves extracting $x- \overline m$ at time $0$. Now, the dm is facing the risk that $Y \in [\overline m, u^*)$. In this case, the DM ends up with a level of reserves strictly below $u^*$ in the downgraded problem; the difference is that, now, it would have been optimal to wait in the downgraded problem starting from the tipping point $Y$ because $Y\leq m <u^*$ and thus $U'(Y) >1$. This aggressive control is therefore suboptimal. Suppose now that the dm were informed of the value of the tipping point $Y$, conditional on it lying in the the interval $[\overline m, m)$. Under this assumption, the optimal control is intuitively to extract at time 0 up to the tipping point $Y$ if $Y \in [\overline m, m)$ and then to act optimally in the downgraded problem, or to extract up to $\overline m$ if $Y < \overline m$ and then apply the control which reflects $x$ at the boundary $b$. In that hypothetical situation, the resulting payoff would be
\begin{align*}
{\EE}_{x,m} \hskip 0.5mm & [[U(Y) +x - Y]\ind_{\{ Y\in (\overline m, m)\}}] +(x - \overline m) F(\overline m)   +W(\overline m, \overline m)
\\
&=   \int_{\overline m}^{m} [U(s) + x - s] f(s) \, \mathrm ds+(x - \overline m) F( \overline m) + W(\overline m, \overline m)
\\
& = (x-m)F(m)+\int_{\overline m}^m [U(s) f(s) + F(s)] \, \mathrm ds +W(\overline m, \overline m),
\end{align*}
which yields \eqref{d12} as $W(\overline m, \overline m) = \frac{\mu(\overline m)}{r} \,F(\overline m)$. Turning back to the tipping-point problem, in which the dm does not know the exact value of $Y$, this heuristic argument suggests that, when $u^* > \overline m$, the dm should aim at extracting ``as fast as possible'' as long as $m \in [\overline m, m)$, effectively ``gliding'' along the diagonal, so as to learn the value of $Y$ as soon as possible in case it lies in $[\overline m, m)$. This, of course, raises the question of whether an optimal control exists at all, a question to which we now turn.

\subsubsection{Optimal and $\varepsilon$-Optimal Controls}

This section construct, for the case $u^*\leq \overline m$, a control $L^*$ that is optimal in the sense that the function $W$ defined in Proposition \ref{verif'} satisfies $W(x,m)=V(x,m;L^*)$ for all $(x,m) \in \CJ$, and, for the case $u^*>\overline m$, a family of controls $L^{*\varepsilon}$ indexed by $\varepsilon>0$ such that $W(x,m)=  \lim_{\varepsilon \to 0}V(x,m;L^{*\varepsilon})$ for all $(x,m) \in \CJ$. This shows in both cases that $W$ is indeed the optimal value function of the tipping-point problem \eqref{pb}.

The proof relies on the following result.

\begin{proposition} \label{skob}
Let $b$ be the solution of the Cauchy problem \eqref{ode1} characterized in Proposition {\rm \ref{lemedo}}. For each $(x,m)\in \mathcal{J}$ and any stopping time $\tau,$ there exists a unique solution $(X,M,L)$ defined on $\{\tau<\infty\}$ on the time interval $[\tau,\tau_0]\cap [\tau,\infty)$ to the reflected sde
\begin{align}
X_{\tau-}&=x, \; M_{\tau-}=m, \; L_{\tau-}=0, \label{sko00}
\\
\mathrm dX_t &= \mu(X_t) \, \mathrm dt + \sigma (X_t) \, \mathrm dB_t - \mathrm dL_t,\; M_t = m \wedge \inf_{\tau \leq s \leq t} X_s, \quad t \geq \tau, \label{sko0}
\\
X_t & \in [M_t, b(M_t)]\cap [0,\overline m] \;\; \mathbb P\text{-a.s.}, \quad t \geq \tau, \label{sko1}
\\
\int_{[\tau,t]}&\, \ind_{\{X_s < b(M_s) \}} \, \mathrm dL_s = 0 \;\; \mathbb P\text{-a.s.}, \quad  t\geq \tau. \label{sko2}
\end{align}
\end{proposition}

\vskip 3mm

Proposition \ref{skob} is related to the existence and uniqueness of solutions of sdes with reflecting boundary conditions for a domain that has a corner at which more than one oblique direction is allowed (see \cite{Dupuis}). In that respect, it should be noted that our condition $b'(\overline m)<1$ is relevant for obtaining a unique \textit{strong} solution to \eqref{sko00}--\eqref{sko2}; a \textit{weak} solution to \eqref{sko00}--\eqref{sko2} is known to exist for reflected sdes in cusps, which would arise if we had $b'(\overline m)=1$ (see \cite{Costantini}). Because we have not been able to find a reference that exactly fits our setting, we provide a proof in Appendix \ref{AppendixProofs}.

\paragraph{Case 1: $\boldsymbol{u^* \leq \overline m}$}

The following result holds.

\begin{proposition}\label{identifyValue}
If $u^* \leq \overline m,$ then the function $W$ defined by \eqref{d1}--\eqref{d12} is the value function of the singular control problem \eqref{markovv}. Moreover$,$ for any initial condition $(x,m) \in \mathcal J,$ letting $(X^{L^*},M^{L^*},L^*)$ denote the solution to \eqref{sko00}--\eqref{sko2} given in Proposition {\rm \ref{skob}} with $\tau \equiv 0,$ $L^*$ is an optimal control.
\end{proposition}

\noindent \textbf{Proof.} Because $W\geq V$, it is sufficient to prove that $W$ defined by \eqref{d1}--\eqref{d12} satisfies $W(x,m)=V(x,m;L^*)$ for all $(x,m) \in \CJ$. Let $(X^{L^*},M^{L^*},L^*)$ denote the solution to (\ref{sko00})--(\ref{sko2}) given in Proposition \ref{skob} with $\tau=0$. Notice that \eqref{sko1}--\eqref{sko2} imply that
\begin{align*}
L^*_0=(x-\overline m)\ind_{ (x,m)\in \mathcal J_3}+ [x-b(m)]\ind_{ (x,m) \in \mathcal J_2},
\end{align*}
and that, except for a potential jump at time $0$, the processes $(X^{L^*},M^{L^*},L^*)$ are continuous and $(X^{L^*},M^{L^*}) \in \mathcal J_1$. As in Lemma \ref{verif}, because $W\in \mathcal{R}(\mathcal J)$, we can apply It\^{o}'s formula to the \pagebreak process  $(\mathrm e^{-r(t\wedge T_n \wedge \tau_0)} W(X^{L^*}_{t\wedge T_n\wedge \tau_0},M^{L^*}_{t\wedge T_n\wedge \tau_0}))_{t \geq 0}$, where $T_n$ is defined as in \eqref{localiz}. Taking expectations, we obtain
\begin{align}
W(x,m) & =  \EE_{x,m} \!\left[\mathrm e^{-r (t\wedge T_n \wedge  \tau_0)} W(X^{L^*}_{t\wedge T_n \wedge  \tau_0}, M^{L^*}_{t\wedge T_n \wedge  \tau_0}) \right] \notag
\\
&  \quad -  \EE_{x,m}\!\left[\int_0^{t\wedge T_n \wedge  \tau_0} \mathrm e^{-rs} ({\cal L}-r)W(X^{L^*}_{s},M^{L^*}_{s}) \, \mathrm ds \right]
\allowdisplaybreaks \notag
\\
& \quad - \EE_{x,m} \!\left[\int_{[0, {t\wedge T_n \wedge  \tau_0}]} \mathrm e^{-rs} W_m (M^{L^*}_{s}, M^{L^*}_{s}) \, \mathrm dM_s^{L^*,c} \right] \notag
\\
& \quad +\EE_{x,m} \! \left[ \int_{[0, t\wedge T_n \wedge  \tau_0]} \mathrm e^{-rs} W_x (X^{L^*}_{s}, M^{L^*}_{s})\, \mathrm dL_s^{*,c} \right] \notag
\\
& \quad -  \EE_{x,m} \! \left[ \sum_{0 \le s \le t\wedge T_n \wedge  \tau_0} \mathrm e^{-rs} [W(X^{L^*}_s, M^{L^*}_s) - W(X^{L^*}_{s-} , M^{L^*}_{s-})] \right] \hskip -1mm. \label{RolloTommasi}
\end{align}
Because $(X_s^{L^*},M_s^{L^*}) \in \CJ_1$ for $s \le t\wedge T_n \wedge \tau_0$, we have $({\cal L}-r)W(X_s^{L^*},M_s^{L^*})=0$, and hence the second term on the right-hand side of \eqref{RolloTommasi} vanishes. Next, because the random measure $\mathrm dM_s^{L^*,c}$ only charges $\{s \ge 0 : X_s^{L^*}= M_s^{L^*}\}$, the boundary condition \eqref{diagonal1} yields
\begin{align*}
\EE_{x,m}\!&\left[\int_{[0, {t\wedge T_n \wedge  \tau_0}]} \mathrm e^{-rs} W_m (M^{L^*}_{s}, M^{L^*}_{s}) \, \mathrm dM_s^{L^*,c} \right]
\\
&=\EE_{x,m} \!\left[\int_{[0, {t\wedge T_n \wedge  \tau_0}]} \mathrm e^{-rs} U(M^{L^*}_{s})f(M^{L^*}_{s}) \, \mathrm dM_s^{L^*,c}\right] \hskip -1mm.
\end{align*}
Then, conditions (\ref{sko1})--(\ref{sko2}) together with the fact that $W_x(b(m),m)=F(m)$ for $m\in (0,\overline m]$ imply that
\begin{align*}
\EE_{x,m} \! \left[ \int_{[0, t\wedge T_n \wedge  \tau_0]} \mathrm e^{-rs} W_x (X^{L^*}_{s}, M^{L^*}_{s}) \, \mathrm dL_s^{*,c} \right] \!= \EE_{x,m} \! \left[ \int_{[0, t\wedge T_n \wedge  \tau_0]} \mathrm e^{-rs} F (M^{L^*}_{s}) \, \mathrm dL_s^{*,c} \right] \hskip -1mm.
\end{align*}
Finally, using that the only possible jump in $L^*$ is at time 0, the last term in \eqref{RolloTommasi} is equal to
\begin{align*}
-\, \EE_{x,m} \!& \left[ \sum_{0 \le s \le t\wedge T_n \wedge  \tau_0} \mathrm e^{-rs} [W(X^{L^*}_s, M^{L^*}_s) - W(X^{L^*}_{s-} , M^{L^*}_{s-})] \right]
\\
&= [ W(x,m)-W(\overline m,\overline m)]\ind_{(x,m)\in \mathcal J_3} + [W(x,m)-W(m,b(m))]\ind_{(x,m)\in \mathcal J_2}
\\
&= \{(x-\overline m)F( m)+U(\overline m)[F(m)-F(\overline m)]\}\ind_{(x,m)\in \mathcal J_3} + [(x-m)F(m)]\ind_{(x,m)\in \mathcal J_2}
\\
&= F(M^{L^*}_{0-})\Delta L^*_0 + U(M^{L^*}_0)[F(M^{L^*}_{0-})-F(M^{L^*}_0)],
\end{align*}
where the second equality follows from \eqref{djump} and the fact that $W_x(y,p)=F(y)$ for $(y,p)\in \mathcal J_2$. Gathering the previous equalities, we obtain
\begin{align*}
W(x,m) & =  \EE_{x,m}\!\left[\mathrm e^{-r (t\wedge T_n \wedge  \tau_0)} W(X^{L^*}_{t\wedge T_n \wedge  \tau_0}, M^{L^*}_{t\wedge T_n \wedge  \tau_0}) \right]
\\
& \quad +\EE_{x,m} \! \left[ \int_{[0, t\wedge T_n \wedge  \tau_0]} \mathrm e^{-rs} F (M^{L^*}_{s-}) dL_s^{*} \right]
\\
& \quad - \EE_{x,m} \! \left[\int_{[0, {t\wedge T_n \wedge  \tau_0}]} \mathrm e^{-rs} U (M^{L^*}_{s})f( M^{L^*}_{s}) \, \mathrm dM_s^{L^*,c} \right]
\\
& \quad +  \EE_{x,m} \! \left[U(M^{L^*}_0)[F(M^{L^*}_{0-})-F(M^{L^*}_0)] \right] \hskip -1mm.
\end{align*}
Letting $n, t \to \infty$, and using that $W$ is bounded on $\CJ_1$, we conclude from \eqref{reward} that $W(x,m)=V(x,m;L^*)$. Hence the result. \hfill $\blacksquare$

\paragraph{Case 2: $\boldsymbol{u^* > \overline m}$}

If $(x,m) \in \CJ_3$ and $\varepsilon>0$, we construct a process $(X^{L^{*\varepsilon}},M^{L^{*\varepsilon}},L^{*,\varepsilon})$ corresponding to a control that initially extracts up to $u^*$, then extracts at a rate $\frac{1}{\varepsilon}$ until $X^{L^{*\varepsilon}}$ reaches the level $\overline m$, and finally reflects $X^{L^{*\varepsilon}}$ on the boundary $b$. First, let
\begin{align}
\label{control1} X^{L^{*\varepsilon}}_{0-}:=x, \; M^{L^{*\varepsilon}}_{0-}:=m, \; L^{*\varepsilon}_{0-}:=0
\end{align}
and
\begin{align}\label{control2}
L^{*\varepsilon}_0:=x-m\wedge u^*, \;  X^{L^{*\varepsilon}}_0:=m\wedge u^*, \;  M^{L^{*\varepsilon}}_0:=m\wedge u^*.
\end{align}
Next, let $(X^{\varepsilon}_t)_{t\geq 0}$ be the unique strong solution to
\begin{align} \label{=5}
\mathrm d  X_t^{ \varepsilon} = \!\left[\mu( X_t^{ \varepsilon}) -\frac{1}{\varepsilon}\right] \mathrm dt + \sigma( X_t^{ \varepsilon}) \, \mathrm dB_t,\quad X_0^{\varepsilon} = X^{L^{*\varepsilon}}_0=m\wedge u^* \in [\overline m,  u^*],
\end{align}
and, letting $\tau_y^\varepsilon := \inf \hskip 0.5mm \{ t \geq 0 : X_t^\varepsilon\leq y\}$ for all $y \in [\overline m, m\wedge u^*]$, define
\begin{align}\label{control3}
L^{*\varepsilon}_t:=L^{*\varepsilon}_0+\frac{1}{\varepsilon}\,t, \;X^{L^{*\varepsilon}}_t:=X^\varepsilon_t, \; M^{L^{*\varepsilon}}_t:=m\wedge \inf_{0 \leq s \leq t}X^{\varepsilon}_s, \quad t \in (0,\tau^\varepsilon_{\overline m}] \cap (0,\infty).
\end{align}
Consider finally the solution $(\hat X^\varepsilon, \hat M^\varepsilon, \hat L^\varepsilon)$ of the reflected sde \eqref{sko00}--\eqref{sko2} with initial condition $(\overline m, \overline  m)$ at time $\tau^\varepsilon_{\overline m}$, and define on $\{ \tau^\varepsilon_{\overline m} <\infty\}$
\begin{align}\label{control4}
L^{*\varepsilon}_t:=L^{*\varepsilon}_{\tau^\varepsilon_{\overline m}}+\hat L^\varepsilon_t, \;X^{L^{*\varepsilon}}_t := \hat X^\varepsilon_t, \; M^{L^{*\varepsilon}}_t := \hat M^\varepsilon_t, \quad t\in (\tau^\varepsilon_{\overline m}, \tau_0] \cap (\tau^\varepsilon_{\overline m}, \infty),
\end{align}
where $\tau_0:=\inf \hskip 0.5mm \{t \geq \tau^\varepsilon_{\overline m}: \hat X^\varepsilon_t \leq 0\}$. The following result then holds.

\begin{proposition}\label{identifyValue'}
If $u^*> \overline m,$ then the function $W$ defined by \eqref{d1}--\eqref{d12} is the value function of the singular control problem \eqref{markovv}. If $m \leq \overline m,$ the control $L^*$ defined in Proposition {\rm \ref{identifyValue}} is optimal. If $m > \overline m,$ the family of controls $(L^{*\varepsilon})_{\varepsilon>0}$ defined by \eqref{control1}--\eqref{control2} and \eqref{control3}--\eqref{control4} satisfies $W(x,m)=\lim_{\varepsilon \to 0} V(x,m;L^{*\varepsilon})$.
\end{proposition}

\noindent \textbf{Proof.} In case $m \leq \overline m$ and $x\geq m$, the proof that $W(x,m)=V(x,m)$ and that the control $L^*$ is optimal is similar to the proof of Proposition \ref{identifyValue}.

In case $m > \overline m$ and $x \geq m$, as $W\geq V$, it is sufficient to prove that $W(x,m)=\lim_{\varepsilon \to 0} V(x,m;L^{*\varepsilon})$. Using that $L^{*\varepsilon}$, $X^{L^{*\varepsilon}}$, and $M^{L^{*\varepsilon}}$ are continuous except possibly for a deterministic jump at time $0$, we have
\begin{align}
V(x,m;L^{*\varepsilon})
&= \EE_{x,m} \!\left [ \int_{[0,\tau_0]} \mathrm e^{-rs}  F(M^{L^{*\varepsilon}}_{s-})  \, \mathrm d L^{*\varepsilon}_s   -\int_{[0,\tau_0]} \mathrm e^{-rs} U(M^{L^{*\varepsilon}}_{s}) f(M^{L^{*\varepsilon}}_{s}) \, \mathrm dM^{{L^{*\varepsilon}},c}_s \right] \notag
\\
& \quad+   \EE_{x,m} \!\left [ \sum_{0\leq s \leq \tau_0} \mathrm e^{-rs}  U(M^{L^{*\varepsilon}}_{s})[F(M^{L^{*\varepsilon}}_{s-}) - F({M^{L^{* \varepsilon}}_s})] \right] \notag
\\
& = \EE_{x,m} \! \left [ \int_{[0,\tau^\varepsilon_{\overline m}]} \mathrm e^{-rs}  F(M^{L^{*\varepsilon}}_{s-})  \, \mathrm d L^{*\varepsilon}_s
-\int_{[0,\tau^\varepsilon_{\overline m}]} \mathrm e^{-rs} U(M^{L^{*\varepsilon}}_{s}) f(M^{L^{*\varepsilon}}_{s}) \, \mathrm dM^{{L^{*\varepsilon }},c}_s \right ] \notag
\\
&\quad + \EE_{x,m} \! \left [ \int_{[\tau^\varepsilon_{\overline m}, \tau_0]} \mathrm e^{-rs}  F(M^{L^{*\varepsilon}}_{s-})  \, \mathrm d L^{* \varepsilon}_s -\int_{[\tau^\varepsilon_{\overline m}, \tau_0]} \mathrm e^{-rs} U(M^{L^{*\varepsilon}}_{s}) f(M^{L^{*\varepsilon}}_{s}) \, \mathrm dM^{{L^{*\varepsilon}},c}_s \right ] \allowdisplaybreaks \notag
\\
& \quad+  U(M^{L^{*\varepsilon}}_{0})[F(M^{L^{*\varepsilon}}_{0-}) - F({M^{L^{*\varepsilon}}_0})]. \label{Ben1}
\end{align}
From \eqref{control4}, using It\^{o}'s formula and arguing exactly as in the proof of Proposition \ref{identifyValue}, we obtain that, on $\{\tau^\varepsilon_{\overline m} < \infty\}$, we have $\mathbb P$-a.s.
\begin{align}
&\mathrm e^{-r \tau^\varepsilon_{\overline m}}W(\overline m, \overline m) \notag
\\
& = \EE_{x,m}\! \left[\int_{[\tau^\varepsilon_{\overline m}, \tau_0]} \mathrm e^{-rs}  F(M^{L^{*\varepsilon}}_{s-})  \, \mathrm d L^{*\varepsilon}_s
-\int_{[\tau^\varepsilon_{\overline m}, \tau_0]} \mathrm e^{-rs} U(M^{L^{*\varepsilon}}_{s}) f(M^{L^{*\varepsilon}}_{s}) \, \mathrm dM^{{L^{* \varepsilon}} ,c}_s \! \mid \!\CF_{\tau^\varepsilon_{\overline m}} \right] \label{Ben2}
\end{align}
Using the same arguments as in the proof of Proposition \ref{markovprop}, we obtain
\begin{align}
\EE_{x,m} \! \left [   -\int_{[\tau^\varepsilon_{\overline m}, \tau_0]} \mathrm e^{-rs} U(M^{L^{*\varepsilon}}_{s}) f(M^{L^{*\varepsilon}}_{s}) \, \mathrm dM^{{L^{*\varepsilon}},c}_s \right ] \! = \EE_{x,m}\! \left[ \int_{\overline m}^{m\wedge u^*} \mathrm e^{-r \tau^\varepsilon_y}U(y)f(y) \, \mathrm dy\right] \hskip -1mm . \label{Ben3}
\end{align}
The standard change of variables formula for Stieltjes integrals (\cite[Chapter 0, Proposition 4.9]{RevuzYor}), yields, $\mathbb P$-a.s.
\begin{align}
\int_{[0, \tau^{{\varepsilon} }_{\overline m}]} \mathrm e^{-rs} F(M_{s-}^{L^{*\varepsilon}})\,\mathrm dL^{*\varepsilon}_s & =(x-m\wedge u^*)F(m) +\int_{[0, \tau^{{\varepsilon} }_{\overline m}]} \mathrm e^{-rs} F(M_{s-}^{L^{*\varepsilon}})\, \mathrm dL^{*\varepsilon,c}_s \notag
\\
& = (x-m\wedge u^*)F(m)+ \int_0^\infty \mathrm e^{-ry\varepsilon} F(M_{y\varepsilon}^{L^{*\varepsilon}}) \ind_{\{y\varepsilon \leq \tau^{{ \varepsilon }} _{\overline m} \}}\, \mathrm dy . \label{Ben4}
\end{align}
Gathering \eqref{Ben1}--\eqref{Ben4}, we obtain
\begin{align}
&V(x,m;L^{*\varepsilon}) \notag
\\
&=  (x-m\wedge u^*)F(m)+  U(m\wedge u^*)[F(m) - F(m\wedge u\*)] \notag
\\
&\quad +\EE_{x,m} \! \left [\int_0^\infty \mathrm e^{-ry\varepsilon} F(M_{y\varepsilon}^{L^{*\varepsilon}}) \ind_{\{y\varepsilon \leq \tau^{{ \varepsilon}}_{\overline m} \}}\, \mathrm dy+ \int_{\overline m}^m \mathrm e^{-r \tau^\varepsilon_y} U(y) f(y) \, \mathrm dy + \mathrm e^{-r \tau^{{\varepsilon} }_{\overline m}} W(\overline m,\overline m) \right ]\hskip -1mm.  \label{eqVespilon}
\end{align}
To conclude the proof, we need the following lemma.
\begin{lemma} \label{forepsi}
For all $y\in [ \overline m , m \wedge  u^*),$ $\tau^{\varepsilon}_{y} \stackrel{{\PP}}{\to } 0$ as $\varepsilon \to 0$.
\end{lemma}

\noindent \textbf{Proof.} Consider the process $ X^\varepsilon$ defined by (\ref{=5}). From the time change formula for It\^{o} integrals (see, for instance, \cite[Theorem 8.5.7]{Oksendal}), the process $\tilde X^{\varepsilon}$ defined by $\tilde X^{\varepsilon}_t\equiv X_{t\varepsilon} ^ \varepsilon$ for all $t\geq 0$ satisfies the sde
\begin{align*}
\mathrm d \tilde X_t^\varepsilon = - \mathrm dt  + \varepsilon \mu(\tilde X_t^\varepsilon) \, \mathrm dt + \sqrt{\varepsilon} \sigma(\tilde X_t^ \varepsilon) \, \mathrm d\tilde B_t, \quad \tilde X_0^{\varepsilon} = m \wedge u^* \in (\overline m, u^*],
\end{align*}
where $\tilde B:=(\tilde B_t)_{t \geq 0}$ is an $({\cal F}_{ t \varepsilon})_{t \geq 0}$-Brownian motion. We claim that
\begin{align} \label{eq11}
\EE \! \left[ \sup_{0 \leq s\leq T} \big|\tilde X_s^{\varepsilon} - (m \wedge u^*-s)\big|^2 \right ] \! \leq  \varepsilon (K_1 \varepsilon + K_2) \text{ for all } T> m\wedge u^*-\overline m,
\end{align}
where $K_1$ and $K_2$ are two positive constants. First, we have
\begin{align}
&\EE \! \left[ \sup_{0 \leq s\leq T} \big|\tilde X_s^{\varepsilon} - (m \wedge u^*-s)\big|^2 \right ] \notag  \allowdisplaybreaks
\\
&\leq 2 \,\EE \!\left [  \sup_{0 \leq s\leq T} \left| \int_0^s \varepsilon \mu(\tilde X_u^\varepsilon) \, \mathrm du \right|^2 \right ] + 2 \, \EE \! \left [  \sup_{0 \leq s\leq T} \left| \int_0^s \sqrt \varepsilon \sigma(\tilde X_u^\varepsilon) \, \mathrm d\tilde B_u \right |^2 \right ]\hskip -1mm . \label{Ben5}
\end{align}
Next,
\begin{align}
\EE \,\left [  \sup_{0 \leq s \leq T} \left| \int_0^s \sqrt \varepsilon \sigma(\tilde X_s^\varepsilon) \, \mathrm d\tilde B_s \right|^{2} \right ] &\leq C_1  \EE \! \left [\int_0^T \varepsilon \sigma^2 (\tilde X_s^\varepsilon) \, \mathrm ds  \right ]\nonumber
\\
& \leq   \varepsilon C_1 T \, \EE \!\left [K \!\left (1 + \sup_{0 \le s\le T} \big| \tilde X_s^\varepsilon \big|^{2} \right )\right ]\nonumber
\\
& \leq   \varepsilon C_1  C_2 T(1 + m^{2}) \,\mathrm e^{C_2T}, \label{6-1}
\end{align}
for some positive constants $K$, $C_1$, and $C_2$, where the first inequality follows from the Burkholder--Davis--Gundy inequalities (\cite[Chapter 3, \S3, Theorem 3.28]{KaratzasShreve}), the second \linebreak inequality follows from $\sigma$ being Lipschitz on $(0, \infty)$, and the third inequality follows from \cite[Chapter 5, \S3, Problem 3.15]{KaratzasShreve}. Last, and similarly,
\begin{align}
\EE \! \left [  \sup_{0 \leq s\leq T} \left| \int_0^s \varepsilon \mu(\tilde X_s^\varepsilon) \, dt \right|^2 \right ] & \leq  \varepsilon^2 T^2 \, \EE  \! \left [K \! \left ( 1 + \sup_{0 \leq s\leq T} \big|\tilde X_s^\varepsilon \big|^2\right )\right ] \leq  \varepsilon^2 T^2 C_3 (1 + m^2) \, \mathrm e^{C_3T}. \label{6-2}
\end{align}
for some positive constants $K$ and $C_3$. The inequality (\ref{eq11}) then follows from (\ref{Ben5})--(\ref{6-2}). We deduce from (\ref{eq11}) that, letting $\tilde \tau^{\varepsilon}_{y} \:= \inf\{t\geq 0 : \tilde X_t^\varepsilon \leq y\} $, we have, for each $\delta \in (0, T-(m\wedge u^*-\overline m))$,
\begin{align} \label{top}
\PP\hskip 0.3mm [T\wedge \tilde \tau^{\varepsilon}_{y}  - (m\wedge u^* - y) >\delta] \leq \PP \!\left[\sup_{0 \leq s\leq T} \big |\tilde X_s^\varepsilon - (m\wedge u^*-s)\big|> \delta \right] \leq \frac{\varepsilon (K_1 \varepsilon + K_2)}{\delta^2},
\end{align}
where the first inequality in (\ref{top}) follows from the fact that $\sup_{0 \leq s \leq T} \big |\tilde X_s^\varepsilon - (m \wedge u^*-s)\big|\leq  \delta $ $\mathbb P$-a.s.\! implies $T \wedge \tilde \tau^{\varepsilon}_{y} \in [ m\wedge u^*- y - \delta,  m \wedge u^* - y + \delta]$, and the second inequality follows from (\ref{eq11}) using Chebyshev's inequality. Finally, it follows from (\ref{top}) that $ \tilde \tau^{\varepsilon}_{y}  \stackrel{{ \PP}}{\to} m \wedge u^*- y$ as $\varepsilon \to 0$, and thus, in turn, that $\varepsilon   \tilde \tau^{\varepsilon}_{y} \stackrel{{\PP}} {\to } 0$ as $\varepsilon \to 0$. This concludes the proof as $ \tilde \tau^{\varepsilon}_{y}=\frac{ \tau^\varepsilon_{y}}{\varepsilon}$  ${\PP}$-a.s. The result follows. \hfill $\blacksquare$

\bigskip

Returning to the proof of Proposition \ref{identifyValue'}, we now use Lemma \ref{forepsi} to study the convergence of \eqref{eqVespilon} as $\varepsilon \to 0$.

By Fubini's Theorem and using dominated convergence, we have
\begin{align}
\lim_{\varepsilon \to 0}\,\EE_{x,m} \! \left [ \int_{\overline m}^{m\wedge u^*} \mathrm e^{-r \tau^\varepsilon_y} U(y) f(y) \, \mathrm dy \right ] \!
&=\lim_{\varepsilon \to 0} \int_{\overline m}^{m\wedge u^*} \EE_{x,m} \! \left [\mathrm e^{-r \tau^\varepsilon_y}\right ] U(y) f(y) \, \mathrm dy \notag
\\
&=\int_{\overline m}^{m\wedge u^*} \lim_{\varepsilon \to 0}\,\EE_{x,m} \! \left [ \mathrm e^{-r \tau^\varepsilon_y}\right ] U(y) f(y) \, \mathrm dy \notag
\\
& = \int_{\overline m}^{m\wedge u^*} U(y) f(y) \, \mathrm dy, \label{conv1}
\end{align}
where the last equality \pagebreak follows from the fact that the random variable $\mathrm e^{-r \tau^\varepsilon_y}$ is bounded and, by Lemma \ref{forepsi}, converges in probability to $1$ as $\varepsilon \to 0$, so that $\lim_{\varepsilon \to 0} \EE_{x,m} \hskip 0.3mm [\mathrm e^{-r \tau^\varepsilon_y}]= 1$. Likewise,
\begin{align}\label{conv2}
\lim_{\varepsilon \to 0} \, \EE_{x,m} \! \left[\mathrm e^{-r \tau^{{\varepsilon} }_{\overline m}} W(\overline m,\overline m) \right ] \! = W(\overline m, \overline m).
\end{align}
The bulk of the proof consists in showing that
\begin{align} \label{r1}
\liminf_{\varepsilon \to  0}\,\EE_{x,m} \! \left [\int_0^\infty \mathrm e^{-ry\varepsilon} F(M^{L^{*\varepsilon}}_{y\varepsilon}) \ind_{\{y\varepsilon \leq \tau^{\varepsilon}_{\overline m} \}}\, \mathrm dy\right ] \! \geq \int_{\overline m}^{m\wedge u^*} \!F(y) \, \mathrm dy.
\end{align}
As observed in the proof of Lemma \ref{forepsi}, $  \tilde \tau^{\varepsilon}_{\overline m}  =\frac{ \tau^\varepsilon_{\overline m}}{\varepsilon}$  $\mathbb P$-a.s. Thus
\begin{align}
\EE_{x,m} \!  \left [\int_0^\infty \mathrm e^{-ry\varepsilon} F(M_{y\varepsilon}^{L^{*\varepsilon}}) \ind_{\{y\varepsilon \leq \tau^{ \varepsilon }_ {\overline m} \}}\, \mathrm dy\right ] & =  \EE_{x,m} \! \left [\int_0^{\tilde \tau^{\varepsilon}_{\overline m}}\mathrm e^{-r\varepsilon  y} F(M^{L^{* \varepsilon}}_{y\varepsilon})\,\mathrm dy  \right ]\nonumber
\\
& \geq  \EE_{x,m} \! \left [\int_0^{\tilde \tau^{\varepsilon}_{\overline m}\wedge (m\wedge u^* - \overline m)} \mathrm e^{-r\varepsilon  y} F(M^{L^{* \varepsilon}}_{y\varepsilon})\,\mathrm dy  \right ]\hskip -1mm. \label{ep3}
\end{align}
The right-hand side of \eqref{ep3} can be decomposed as follows:
\begin{align}
&\EE_{x,m}  \! \left [\int_0^{\tilde \tau^{\varepsilon}_{\overline m}\wedge (m\wedge u^*-\overline m)} \mathrm e^{-r\varepsilon  y} F(M^{L^{* \varepsilon }}_{y\varepsilon})\, \mathrm dy  \right ] \notag
\\
& \;\; \qquad \qquad \qquad \qquad  = \int_0^{m\wedge u^* - \overline m} \mathrm e^{-r\varepsilon y} F(m\wedge u^*-y) \, \mathrm dy \notag
\\
&  \;\; \qquad \qquad \qquad \qquad \quad- \,\EE_{x,m} \! \left [\int_{\tilde \tau^\varepsilon_{\overline m} \wedge (m \wedge u^*- \overline m)}^{m\wedge u^* - \overline m}  \mathrm e^{-r\varepsilon y} F(m\wedge u^*- y) \, \mathrm dy \right ] \notag
\\
&  \;\; \qquad \qquad \qquad \qquad \quad + \,\EE_{x,m} \! \left [\int_0^{\tilde \tau^\varepsilon_{\overline m} \wedge (m\wedge u^*  - \overline m) } \mathrm e^{-r\varepsilon y} [F(M^{L^{*\varepsilon}}_{y\varepsilon}) - F(m\wedge u^* - y)]\, \mathrm dy\right ] \hskip -1mm. \label{ABC}
\end{align}
Let us examine the three terms in the last expression in turn.

\medskip

(1) For the first one, we have
\begin{align}
\label{rr2}\lim_{\varepsilon \to 0} \, \int_0^{m\wedge u^* - \overline m} \mathrm e^{-r\varepsilon y} F(m \wedge u^*-y) \, \mathrm dy  =  \int_{\overline m}^{m\wedge u^*} F(y) \, \mathrm dy
\end{align}
by bounded convergence.

\medskip

(2) For the second one, we have
\begin{align*}
\left |\int_{\tilde \tau^\varepsilon_{\overline m} \wedge (m \wedge u^*- \overline m)}^{m \wedge u^* - \overline m} \mathrm e^{-r\varepsilon y} F(m \wedge u^*- y) \, \mathrm dy \right | \leq | m \wedge u^* - \overline m - \tilde\tau^{\varepsilon}_{\overline m}|\ind_{\tilde \tau^{\varepsilon} _{\overline m} < m \wedge u^* - \overline m}.
\end{align*}
From the proof of Lemma \ref{forepsi}, the right-hand side of the latter inequality is a bounded random variable which converges in probability to $0$ as $\varepsilon \to 0$. Thus, its expectation tends to $0$ as $\varepsilon \to 0$, which implies
\begin{align} \label{rr2'}
\lim_{\varepsilon \to 0} \, \EE_{x,m} \! \left [\int_{\tilde \tau^\varepsilon_{\overline m} \wedge (m \wedge u^*- \overline m)}^{m\wedge u^* - \overline m} \mathrm e^{-r\varepsilon y} F(m \wedge u^*- y) \, \mathrm dy \right ] \!=0.
\end{align}

(3) For the third one, we have, for some positive constant $C$,
\begin{align}
\EE_{x,m} \! & \left [\left|\int_0^{\tilde\tau^\varepsilon_{\overline m} \wedge (m\wedge u^* - \overline m) } \mathrm e^{-r\varepsilon y} [F(M^{L^{* \varepsilon}}_{y\varepsilon}) - F(m \wedge u^* - y)]\, \mathrm dy \right|\right ]\nonumber
\\
& \leq C \,\EE_{x,m} \! \left [\int_0^{\tilde\tau^\varepsilon_{\overline m} \wedge (m \wedge u^* - \overline m) } \! \left|\inf_{0 \leq t \leq y} \tilde X_t^\varepsilon - \inf_{0 \leq t \leq y} (m\wedge u^*-t) \right | \mathrm dy\right ] \nonumber
\\
&\leq C(m\wedge u^*- \overline m) \, \EE_{x,m} \! \left [\sup_{0 \leq y \leq \tilde \tau^\varepsilon_{\overline m}\wedge (m \wedge u^* - \overline m)]}\left|\inf_{0 \leq t\leq y } \tilde X_t^\varepsilon - \inf_{0 \leq t\leq y} (m\wedge u^*-t) \right|\right ]\nonumber
\\
&\leq C(m \wedge u^*- \overline m) \,\EE_{x,m} \! \left [\sup_{0 \leq y \leq T} \big|\tilde X_y^\varepsilon - (m\wedge u^*-y)\big|\right ]\nonumber
\\
&\leq C(m\wedge u^* - \overline m) \, \sqrt{\EE_{x,m} \! \left [\sup_{0 \leq y \leq T} \big|\tilde X_y^\varepsilon - (m\wedge u^*-y) \big|^2\right ]}, \label{1805}
\end{align}
where the first inequality follows from $F$ being Lipschitz and from the fact that $M_{y\varepsilon}^{L^{*\varepsilon}} = \inf_{0 \leq t\leq y \varepsilon} X_t^{L^{*\varepsilon}}= \inf_{0 \leq t\leq y }\tilde X_t^{L^{*\varepsilon}}$, and the last inequality follows from Jensen's inequality. Together with (\ref{eq11}), (\ref{1805}) implies
\begin{align} \label{rr3}
\lim_{\varepsilon \to 0}\, \EE_{x,m} \! \left [ \left|\int_0^{\tilde \tau^\varepsilon_{\overline m} \wedge (m \wedge u^* - \overline m) } \mathrm e^{-r\varepsilon y} [F(M^{L^{*\varepsilon}}_{y\varepsilon}) - F(m\wedge u^* - y)]\, \mathrm dy \right|\right ]\!=0.
\end{align}
The inequality \eqref{r1} is then a direct consequence of (\ref{ABC}), (\ref{rr2}),  (\ref{rr2'}), and (\ref{rr3}).

\medskip

To complete the proof, notice, using \eqref{d12}, that
\begin{align*}
W(x,m)&= W(\overline m, \overline m)+ ( x - m) F(m) +\int_{\overline m}^m [U(s)f(s)+F(s)] \, \mathrm ds
\\
&=W(\overline m, \overline m)+ (x - m\wedge u^*) F(m)
\\
&\quad +U(m\wedge u^*)[F(m)-F(m\wedge u^*)] +\int_{\overline m}^{m\wedge u^*} [U(s)f(s)+F(s)] \, \mathrm ds,
\end{align*}
where the second equality is obtained in the same way as \eqref{djump}. Using \eqref{conv1}--\eqref{r1}, we have
\begin{align*}
\liminf_{\varepsilon \to 0} \, V(x,m,L^{*\varepsilon}) \geq W(x,m).
\end{align*}
From Proposition \ref{verif'}, $W(x,m) \geq V(x,m,L^{*\varepsilon})$ for all $\varepsilon>0$, which proves that
\begin{align*}
\lim_{\varepsilon \to 0} \, V(x,m,L^{*\varepsilon}) = W(x,m).
\end{align*}
Hence the result. \hfill $\blacksquare$

\section{Concluding Remarks}

Our contribution in this paper is threefold.

First, we introduced a new class of optimal extraction problems, in which a dm with randomly fluctuating reserves also faces the risk of reserves crossing a tipping point whose exact location is unknown to him ex ante, and below which he obtains a downgraded continuation value. As we argued in the introduction, this class of problems is relevant for many applications in ecology and economics.

Second, we showed that any such optimal extraction problem is amenable to a Markovian formulation, in which the two state variables are the current level of reserves and the minimum level of reserves reached so far. The dm's reward functional in the resulting two-dimensional singular control problem features integrals with respect to future extraction increments and future decrements of the running minimum. We gave a general verification lemma for this new class of singular control problems.

Third, under natural assumptions on the dm's continuation value once the tipping point has been reached, we obtained a complete characterization of the dm's value function and of his optimal extraction policy whenever it exists. Key to this characterization is the identification of an optimal extraction boundary that is solution to an ode.

By providing a complete solution for a relevant class of singular control problems for a diffusion and its running minimum, we hope that this paper will open new avenues of research in stochastic control theory and its applications to a broad spectrum of problems in optimal resource management.

\section*{Appendix}

\renewcommand{\thesection}{A}

\numberwithin{equation}{section}

\setcounter{equation}{0}

\setcounter{subsection}{0}

\newtheorem{applem}{Lemma}[section]

\subsection{A Class of Smooth Functions on $\CJ$}\label{Appendix:class}

We define $\mathcal{R}(\CJ)$ as the set of functions $w \in \CC^0(\CJ)$ such that there exists a finite sequence $0<m_1< \ldots <m_k$ for $k\geq 1$, with $w\in \CC^{2,1}(A_i)$ for all $1 \leq i \leq k+1$, where
\begin{align*}
A_1 &: =\{(x,m) \in \CJ : 0 < m \leq m_1\}
\\
A_i &: =\{(x,m) \in \CJ : m_{i-1}\leq m \leq m_i\} \text{ for all } 1< i \leq k,
\\
A_{k+1} &: =  \{(x,m) \in \CJ : m \geq m_k \},
\end{align*}
with the convention that, for any nonempty set $A\subset (0,\infty)^2$, $w\in \CC^{2,1}(A)$ if there exists a $\CC^{2,1}$ function defined on a neighborhood of $A$ in $(0,\infty)^2$ that coincides with $w$ on $A$. For each $w\in \mathcal{R}(\CJ)$, the partial derivatives $w_m$, $w_x$, and $w_{xx}$ are well-defined on $\{(x,m) \in \CJ :m>0 \text{ and } m\neq m_i \text{ for all } 1 \leq i \leq k\}$. Moreover, on $\{(x,m) \in \CJ :m>0 \text{ and } m= m_i \text{ for some } 1 \leq i \leq k\}$, the left-hand limits $w_m(x,m-)$, $w_x(x,m-)$, and $w_{xx}(x,m-)$ are always well-defined, and we shall hereafter use the convention that the quantities $w_m(x,m)$, $w_x(x,m)$, and $w_{xx}(x,m)$ correspond to these limits whenever $m=m_i$ for some $1\leq i \leq k$. As a result, we can justify the application of It\^o's formula for functions of class $\mathcal{R}(\CJ)$ by applying an appropriate version of It\^o's formula on every set $A_i$, $1 \leq i \leq k+1$. Specifically, for each $L\in \mathcal A(x)$ and for any such $i$, let $\sigma_i$ denote the first exit time of $A_i$ by the process $(X^L,M^L)$; because the process $M^L$ is nonincreasing, $\sigma_1=\tau_0$ and $\sigma_i \leq \sigma_{i+1}$ for all $1 \leq i \leq k$. Moreover, as $w\in \mathcal{R}(\CJ)$, we can apply the generalized version of It\^o's formula in \cite[Theorem 3.1]{Peskirlevampire} to the process $(\mathrm e^{-r (t\wedge \tau_0\wedge T_n)} w(X^L_{t\wedge \tau_0\wedge T_n}, M^L_{t \wedge \tau_0\wedge T_n}))_{t \geq 0}$, where the stopping time $T_n$ is given by (\ref{localiz}) for all $n \in \mathbb N$. Indeed, if $(X^L_0,M^L_0)\in A_{i_0}$, then the formula applies on $[0, t \wedge \sigma_{i_0} \wedge T_n)$ and extends to $t \wedge\sigma_{i_0}\wedge T_n$ by adding a potential jump at $t \wedge \sigma_{i_0}\wedge T_n$; moreover, as the process $M^L$ is nonincreasing, $(X^L,M^L)$ will successively exit only finitely many sets $A_{i_0}, \ldots, A_1$ until $\sigma_1= \tau_0$, so that the general formula follows by concatenation.

\subsection{Auxiliary Lemmas} \label{AuxLemmas}

\begin{applem} \label{forodegenegood}
The following holds$:$
\begin{itemize}

\item[(i)]

The mapping $m \mapsto \frac{N(m,m)}{G(m)D(m,m)}$ is decreasing on $[0, \overline x].$

\item[(ii)]

The mapping $x \mapsto  \frac{N(x,m)}{G(x)D(x,m)}$ is decreasing on $[m, \eta(m)]$ for all $m\in (0, \overline x]$.

\end{itemize}
\end{applem}

\noindent \textbf{Proof.} (i) We have
\begin{align*}
\frac{N(m,m)}{G(m)D(m,m)}= \frac{\mu(m)-rU(m)}{r-\mu'(m)},
\end{align*}
the derivative of which with respect to $m$ has the same sign as
\begin{align*}
[\mu'(m)-rU'(m)] [r-\mu'(m)] +[\mu(m)-rU(m)] \mu''(m),
\end{align*}
which is negative on $[0, \overline x]$ because $\mu' -rU' \leq \mu'-r <0$ by Assumptions \ref{ass:A2}--\ref{ass:A6}, $\mu''\leq 0$ by Assumption \ref{ass:A12}, and $\mu-rU \geq 0 $, and $\mu-rU\geq 0$ on $[0, \overline x]$ by Lemma \ref{lem:existence_barx}.

\medskip

(ii) The derivative of this mapping with respect to $x$ has the same sign as
\begin{align*}
N_x(x,m) G(x)D(x,m) -  N(x,m) [G'(x)D(x,m)+ G(x)D_x(x,m)].
\end{align*}
From Step 2 of the proof of Proposition \ref{bench}, $N(x,m) \geq 0$ for all $m \in [0, \overline x]$ and $x\in [m, \eta(m)]$. Moreover, $G'(x) = -{\mu''(x) \over r} \geq 0$ for all $x \geq 0$ by Assumption \ref{ass:A12}, and $D(x,m) >0$ for all $x \geq m \geq 0$ by \eqref{aux1}. Hence, as $G(x)=1-{\mu'(x) \over r}>0$ for all $x \geq 0$ by Assumption \ref{ass:A2}, it is sufficient to prove that
\begin{align}\label{ineq_lem_1}
N_x(x,m) D(x,m) - N(x,m)D_x(x,m) < 0
\end{align}
for all $m\in [0, \overline x]$ and $x\in [m, \eta(m)]$. Recalling that $D(x):=D(x,x)$, we have
\begin{align*}
N(x,m)= \phi(x)\psi(m)-\psi(x)\phi(m)+ \frac{\mu(x)}{r}\,D(x,m) - D(x)U(m)
\end{align*}
by \eqref{aux2}, and hence
\begin{align*}
N_x(x,m)= \!\left[\frac{\mu'(x)}{r}-1\right]\! D(x,m) + \frac{\mu(x)}{r}\, D_x(x,m)  - D'(x)U(m)
\end{align*}
by \eqref{aux1}, where
\begin{align*}
D'(x)=\frac{ - 2\mu(x)}{\sigma^2(x)} \, D(x)
\end{align*}
and
\begin{align*}
D_x(x,m) = \frac{2r}{\sigma^2(x)} \left[ \psi(x)\phi(m) -\phi(x)\psi(m) - \frac{\mu(x)}{r}\, D(x,m) \right]
\end{align*}
by \eqref{aux1} again, taking advantage from the fact that both $\psi$ and $\phi$ satisfy \eqref{eq:ODE_phi_psi}. Substituting these equalities in \eqref{ineq_lem_1}, we obtain after rearranging terms
\begin{align*}
N_x(x,m) D(x,m) &- N(x,m)D_x(x,m)
\\
& = \!\left[\frac{\mu'(x)}{r}-1\right] \! [D(x,m)]^2 -\frac{2r}{\sigma^2(x)} \,[\psi(x)\phi(m) -\phi(x)\psi(m)] N(x,m),
\end{align*}
which is negative because $D(x,m)>0$ for all $x \geq m \geq 0$, $\mu'(x)<r$ for all $x \geq 0$, $N(x,m) \geq 0$ for all $m \in [0, \overline x]$ and $x\in [m, \eta(m)]$, and  $\psi(x)\phi(m) -\phi(x)\psi(m)\geq 0$ for all $x\geq m \geq 0$ as $\psi$ is increasing and $\phi$ is decreasing. The result follows. \hfill $\blacksquare$

\begin{applem} \label{regulW}
$W\in {\cal R}({\cal J})$.
\end{applem}

\noindent \textbf{Proof.} Given the definition (\ref{d1})--(\ref{d12}) of $W$, it can be easily verified that $W \in \CC^0({\CJ})$. In the remainder of the proof, we focus on the case where $\overline m < \overline y$; the other cases $\overline m=\overline y$ and $\overline m> \overline y$ can be similarly handled without difficulty. We will show in turn that $W \in \CC^{2,1} (A_3)$, $W \in \CC^{2,1} (A_2)$, and $W \in \CC^{2,1} (A_1)$, where $A_3 := \{(x,m) \in {\cal J} :  m \geq \overline y \}$, $A_2:=\{(x,m) \in{\cal J} : \overline m \leq m\leq \overline y \}$, and $A_1:=\{(x,m) \in{\cal J} :  m\leq \overline m \}$.

\medskip

(1) For $(x,m)\in A_3$, because $f(s)=0$ for $s>\overline y$ and $F(s)=1$ for $s\geq \overline y$, we have
\begin{align*}
W(x,m)& = \frac{\mu(\overline m)}{r} \, F(\overline m)+ ( x - m)  F(m) +\int_{\overline m}^m [U(s)f(s)+F(s)]\, \mathrm ds
\\
&=\frac{\mu(\overline m)}{r}  \, F(\overline m)+ x - m +\int_{\overline m}^{\overline y} [U(s)f(s)+F(s)] \, \mathrm ds+m-\overline y.
\end{align*}
The mapping $(x,m) \mapsto \frac{\mu(\overline m)}{r}  \, F(\overline m)+ x - m +\int_{\overline m}^{\overline y}\, [U(s)f(s)+F(s)] \, \mathrm ds+m -\overline y$ defined on $(0, \infty)^2$ coincides with $W$ on $A_3$ and is $\CC^{2,1}$ on $(0,\infty)^2$ and thus on any neighborhood of $A_3$ in $(0, \infty)^2$. Hence $W \in \CC^{2,1}(A_3)$.

\medskip

(2) For $(x,m)\in A_2$, a difficulty arises due to the discontinuity of $f$ at $\overline y$. As in the proof of Proposition \ref{lemedo}, let $\hat f$ denote a Lipschitz positive function on $[0,\infty)$ that coincides with $f$ on $[0,\overline y]$, and let $\hat F(m):=\int_0^m \hat f(t) \, \mathrm dt$ for all $m >0$. The mapping $(x,m)\mapsto \frac{\mu(\overline m)}{r} \, \hat F(\overline m)+ ( x - m) \hat F(m) +\int_{\overline m}^m \, [U(s)\hat f(s)+\hat F(s)] \, \mathrm ds$ defined on $(0, \infty)^2$ coincides with $W$ on $A_2$ and is $\CC^{2,1}$ on $(0,\infty)^2$ and thus on any neighborhood of $A_2$ in $(0, \infty)^2$. Hence $W \in \CC^{2,1} (A_2)$.

\medskip

(3) Finally, we address the case of the set $A_1$. We first claim that $W$ is $\CC^{2,1}$ on the open set $\{(x,m) :  0 < m <\overline m \wedge x \}$. As $\psi$ and $\phi$ are $\CC^2$ on $(0,\infty)$, and $F$ and $b$ are $\CC^1$ on $(0,\overline m)$, the functions (\ref{d1}) and (\ref{d11}) are  $\CC^{2,1}$ on $\mathrm {int} \,\CJ_1$ and $\mathrm {int} \,\CJ_2$, respectively. To prove that the function obtained by pasting together (\ref{d1})--(\ref{d11}) is $\CC^{2,1}$ on $\{(x,m) :  0 < m <\overline m \wedge x \}$, we only need to verify this assertion on a neighborhood of $(m, b(m))$ for all $m \in (0, \overline m)$. This in turn follows from two observations. First, by construction,
\begin{align} \label{reg1}
W_x(b(m)^+,m) = F(m) = W_x (b(m)^-, m) \text{ and } W_{xx}(b(m)^+,m) =0=W_{xx}(b(m)^-,m).
\end{align}
Second, because $W$ is continuous at $(m, b(m))$ for each $m \in (0, \overline m)$, we have
\begin{align*}
\frac{\mu(b(m))}{r} \, F(m)=A(m) \psi(b(m)) + B(m) \phi(b(m)).
\end{align*}
Differentiating this expression and using (\ref{reg1}) yields
\begin{align*}
\left[\frac{\mu(b(m))}{r} \,F(m)\right]' \!\!& = [A(m) \psi'(b(m)) + B(m) \phi'(b(m))] b'(m) + A'(m) \psi(b(m)) + B'(m) \phi(b(m)) \nonumber
\\
& =  F(m)b'(m) + A'(m) \psi(b(m)) + B'(m) \phi(b(m)),
\end{align*}
from which we deduce that
\begin{align*}
W_m(b(m),m^+) = - b'(m)F(m) + \left[\frac{\mu(b(m))}{r} \,F(m)\right]'\!= W_m (b(m), m^-).
\end{align*}
Thus $W$ is $\CC^{2,1}$ on the open set $\{(x,m) :  0 < m <\overline m \wedge x \}$, as claimed. By Proposition \ref{lemedo}, the solution $b$ to the Cauchy problem \eqref{ode2} is defined on $[0, \overline x]$  and is $\CC^1$ on $(0,\overline y \wedge \overline x)$, with  $\overline m < \overline y \wedge \overline x$ thanks to our assumption. It follows that the expressions (\ref{d1})--(\ref{d11}) can be extended to a function that coincides with $W$ on $A_1$ and is $\CC ^{ 2,1}$ on some neighborhood of $A_1$. Hence $W \in \CC^{2,1}(A_1)$. The result follows. \hfill $\blacksquare$

\subsection{Proofs} \label{AppendixProofs}

\noindent \textbf{Proof of Proposition \ref{bench}.} We borrow several arguments from the proof of \cite[Theorem 4.3]{SGL}.

\medskip

(i) If $m \geq \overline x$, then $\mu(m) \leq r U(m)$ by Lemma \ref{lem:existence_barx}. Letting $W^m(x) : =x-m+U(m)$ for $x\geq m$, we have $W^m(m)=U(m)$, $W^{m \prime}= 1$, $W^{m \prime \prime}=0$, and $(\CL-r)W^m \leq 0$ on $[m,\infty)$, where the inequality follows from $\mu'< r$ by Assumption \ref{ass:A2}. A standard verification argument (see \cite[Lemma 3.1]{SGL}) implies $W^m\geq V^m$. As $W^m$ is the payoff associated to any control $L \in \mathcal A^m(x)$ such that $L_0=x-m$, we conclude that $V^m=W^m$, which proves (i).

\medskip

(ii)--(iv) If $m <\overline x$, then $\mu(m) >r U(m)$ by Lemma \ref{lem:existence_barx}. The proof consists of five steps.

\subparagraph{Step 1}

We first claim that, for $m \in [0, \overline x)$, if $N(\eta(m),m)=0$ for some $\eta(m)>m$, then $V^m$ is given by \eqref{eq:formulaVm} and the control $L^m$ defined by \eqref{eq:def_Lm} is optimal.

Indeed, let $W^m$ be given by \eqref{eq:formulaVm}. By construction, $W^m(m)=U(m)$, $(\CL-r)W^m=0$ on $[0,\eta(m))$, and $W^m_x=1$ and $W^m_{xx}=0$ on $(\eta(m),\infty)$. The constants $A^m$ and $B^m$ in \eqref{Aaux}--\eqref{Baux} are constructed so that $W^m_x(\eta(m)-)=1$, and thus $W^m$ is $\mathcal C^1$. Notice that $\psi$ and $\phi$ satisfy
\begin{align}\label{eq:ODE_phi_psi}
u''  =\frac{2}{\sigma^2} \, (r u - \mu u') \text{ on } [0, \infty).
\end{align}
Using \eqref{eq:ODE_phi_psi}, a direct computation shows that $N(\eta(m),m)=0$ implies
\begin{align*}
A^m\psi''(\eta(m))+B^m \phi''(\eta(m))=0,
\end{align*}
so that $W^m$ is $\mathcal C^2$ on $[m,\infty)$, with $W^m_{xx}(\eta(m))=0$. The properties of the solutions to the ode ${\cal L} h - rh =0$ when Assumption \ref{ass:A2} is satisfied, as derived in \cite[Lemma 4.2]{SGL}, imply that $W^m$ is increasing and concave on $[m,\eta(m)]$. As in \cite[Lemma 3.1]{SGL}, a standard verification argument then yields $W^m \geq V^m$. Finally, the existence of the process $L^m$ solution of \eqref{eq:def_Lm} is well-known and follows for example as a special case of Lemma \ref{step1} below, and that $W^m$ coincides with the payoff associated to $L^m$ follows from applying It\^o's formula (see the discussion before \cite[Theorem 4.3]{SGL}). This concludes the proof that $V^m=W^m$. The claim follows.

\subparagraph{Step 2}

We next claim that, if $m \in [0,\overline x]$, then the equation $N(x,m)=0$ admits at most one solution $x \geq m$, denoted by $\eta(m)$. Moreover, $\eta( \overline x) = \overline x$ and, for each $m \in [0,\overline x)$, $\eta (m) >m$ whenever $\eta(m)$ is well-defined.

Indeed, let $m \in [0,\overline x]$ and assume that $N(\eta(m),m)=0$ for some $\eta(m) \geq m$. To prove the uniqueness of $\eta(m)$, notice that
\begin{align*}
&\frac{\partial N}{\partial x}\,(x,m)
\\
&=\!\left[\frac{\mu'(x)}{r}-1\right]\! D(x,m) +\frac{\mu(x)}{r}\,[\psi''(x)\phi(m) -\phi''(x)\psi(m)] -U(m)[\psi''(x) \phi(x)-\phi''(x) \psi(x)].
\end{align*}
From $N(\eta(m),m)=0$ together with \eqref{eq:ODE_phi_psi} and Assumption \ref{ass:A2}, this implies
\begin{align} \label{eq:deriv_g_neq0}
\frac{\partial N}{\partial x}\,(\eta(m),m)=\! \left[\frac{\mu'(\eta(m))}{r}-1\right]\!D(\eta(m),m)<0 ,
\end{align}
which proves that $\eta(m)$ is the unique solution to $N(x,m)=0$ and that $N(x,m)\gtrless 0$ if $x \lessgtr \eta(m)$ for all $x\geq m$. Finally, $\eta( \overline x)= \overline x$ as $N(\overline x, \overline x)=0$ and $\eta(m) >m$ if $m \in [0, \overline x)$ as $N(m,m) = \big[\frac {\mu(m)}{r}-U(m)\big] D(m,m)>0$. The claim follows.

\subparagraph{Step 3}

We now claim that, if $0 \leq m_1<m_2 \leq \overline x$ and $\eta(m_2)$ is well-defined, then $V^{m_1}(x) > V^{m_2}(x)$ for all $x \geq m_2$.

Indeed, let us define a control $\tilde L$ as follows. First, letting $(X^{L^{m_2}},L^{m_2})$ be the solution to \eqref{eq:def_Lm} for $m=m_2$, we let $\tilde L :=L^{m_2}$ on $[0,\tau^{m_2})$, where $\tau^{m_2} :=\inf \hskip 0.5mm \{t \geq 0: X^{L^{m_2}}_t \leq m_2\}$. Second, we let $\tilde L : = L^{*}$ on $[\tau^{m_2},\infty)$, where $L^*$ intuitively reflects $X$ at $\overline x$ starting at time $\tau_{m_2}$. Formally, we consider the unique solution $(X^{L^*}, L^*)$ for $t\geq \tau^{m_2}$ to
\begin{align} \label{eq:def_Lstar}
\left\{ \begin{array} {rcl}
\mathrm dX^{L^*}_t \hskip -2.5mm  &=& \hskip -2mm \mu(X^{L^*}_t)\, \mathrm dt+ \sigma(X^{L^*}_t)\, \mathrm dB_t- \mathrm dL^*_t, \quad X^{L^*}_{\tau ^{m_2}-}=m_2
\\
L^*_t \hskip -2mm & =& \hskip -2mm \sup_{0\leq s \leq t} \big[L^{m_2}_{\tau^{m_2}}+ \int_0^s \mu(X^{L^*}_u) \,\mathrm du + \int_0^s \sigma(X^{L^*}_u) \,\mathrm dB_u - \overline x \big]^+
\end{array}
\right. \hskip -2mm.
\end{align}
The existence of $L^{m_2}$ and $L^*$ follows for example from Lemma \ref{step1} below; notice that $\tilde L$ is continuous except possibly at time $0$. Letting $\tau^{m_1} : =\inf\{t\geq \tau^{m_2} :  X^{\tilde L}_t\leq m_1\}$, we have
\begin{align}
V^{m_1}(x) & \geq \EE_{x} \! \left [\int_{[0,{\tau^{m_1}}]} \mathrm e^{-rs} \, \mathrm d \tilde L_s  + \mathrm  e^{-r \tau^{m_1}}U(m_1)\right ]\notag
\\
& =  \EE_{x} \! \left [\int_{[0,{\tau^{m_2}}]} \mathrm e^{-rs} \, \mathrm dL_s^{m_2} + \int_{(\tau^{m_2},\tau^{m_1}]} \mathrm e^{-rs} \, \mathrm d L^*_s + \mathrm e^{-r \tau^{m_1}} U(m_1) \right ]\hskip -1mm, \label{choli1}
\end{align}
where the first inequality holds  because $\tilde L$ is an admissible process for problem (\ref{aux}) when $m= m_1$, and the equality holds because $\tau^{m_2} < \tau^{m_1}$ $\mathbb P$-a.s.\! on $\{\tau^{m_2}<\infty\}$. Now, notice that $X^{\tilde L} \in [0, \overline x ]$ for $t\in [\tau^{m_2}, \tau^{m_1}]\cap[0,\infty)$ and is continuous on that time interval, and that $X^{\tilde L}_{\tau^{m_i}}=m_i$ $\mathbb P$-a.s.\! on $\{\tau^{m_i}< \infty \}$ for $i=1,2$. Therefore,
\begin{align}
& \EE_{x} \!\left[\int_{(\tau^{m_2},\tau^{m_1}]} \mathrm e^{-rs} \, \mathrm dL^*_s+ \mathrm e^{-r \tau^{m_1}} U(m_1) \right ] \notag \allowdisplaybreaks
\\
&= \EE_{x} \! \left [\int_{(\tau^{m_2},\tau^{m_1}]}\mathrm e^{-rs} \, \mathrm dL^*_s + \mathrm e^{-r \tau^{m_1}} U(X^{\tilde L}_{\tau^{m_1}})\right] \notag
\\
& =  \EE_{x} \! \left [\mathrm e^{-r \tau^{m_2}}U(X^{\tilde L}_{\tau^{m_2}}) \right ] \! +  \EE_{x} \!\left [\int_{(\tau^{m_2},\tau^{m_1}]}\mathrm e^{-rs}(\CL U   - r U)(X^{\tilde L}_s) \, \mathrm ds \right ] \notag
\\
&  \quad  + \, \EE_x \! \left[ \int_{(\tau^{m_2},\tau^{m_1}]} \mathrm e^{-rs}[1-U'(X^{\tilde L}_s)]\, \mathrm dL^{*}_s\right]\notag
\\
&= \EE_{x} \hskip 0.15mm \big [\mathrm e^{-r \tau^{m_2}}U(m_2) \big] + \EE_{x} \! \left [\int_{(\tau^{m_2},\tau^{m_1}]} \mathrm e^{-rs}(\CL U   - r U)(X^{\tilde L}_s) \, \mathrm ds \right ]\notag
\\
&  \quad  + \, \EE_x \! \left[ [1-U'(\overline x)] \int_{(\tau^{m_2},\tau ^{m_1}]} \mathrm e^{-rs} \, \mathrm dL^{*}_s\right]\notag
\\
& >  \EE_{x} \hskip 0.15mm \big [\mathrm e^{-r \tau^{m_2}}U(m_2) \big], \label{choli2}
\end{align}
where the second equality follows from It\^o's formula, the third equality follows from the fact that the random measure $\mathrm dL^*$ only charges $\{s \in (\tau^{m^2}, \tau^{m^1}]: X^{\tilde L}_s= \overline x\}$, and the final inequality follows from the fact that $U'(\overline x)=1$ as $\overline x\geq u^*$ and from \eqref{tard} together with $\PP_x \hskip 0.3mm [\tau_{m_2}<\tau_{m_1}]>0$. Chaining \eqref{choli1}--\eqref{choli2} yields
\begin{align*}
V^{m_1}(x) > \EE_{x} \! \left [\int_{[0,{\tau^{m_2}}]} \mathrm e^{-rs} \, \mathrm dL_s^{m_2} + \mathrm e^{-r \tau^{m_2}} U(m_2) \right ] \! = V^{m_2} (x).
\end{align*}
The claim follows.

\subparagraph{Step 4}

We finally claim that the mapping $m \mapsto \eta(m)$ is well-defined, increasing, and $\mathcal C^1$ on $[0,\overline x]$.

Indeed, recall from Step 1 that $\eta(\overline x)=\overline x$ and $N(x,\overline x)<0$ for all $x > \overline x$. By continuity of $N$, the set
$H:=\{ m \in [0,\overline x] : N(x,m)<0 \text{ for some } x >m\}$ is open in $[0,\overline x]$ and contains $\overline x$. Let $\delta:= \inf \hskip 0.5mm \{ y \in [0, \overline x) : [y, \overline x] \subset H\}$, so that $\delta<\overline x$. Because $N(m,m) >0$ for all $m \in [0, \overline x)$ and $N(\overline x, \overline x)=1$, there exists by Step 1 a unique solution $\eta(m)$ to $N(x,m) = 0$ for all $m \in ( \delta, \overline x]$.
As $N$ is $\CC^1$ on $[0,\infty)^2$, we deduce from \eqref{eq:deriv_g_neq0} together with the implicit function theorem that $\eta$ is $\CC^1$ on $(\delta,\overline x]$. Let us define the function $h(x,m) : = V^m(x) - \frac{\mu(x)}{r}$ for $m \in (\delta,\overline x)$ and $x\geq m$. By Step 1, \eqref{eq:formulaVm} holds for $m\in (\delta, \overline x]$, and thus $h(\eta(m),m) =0$. Moreover, because $\eta$ is $\CC^1$, so is $h$ by \eqref{eq:formulaVm}. Differentiating, we obtain
\begin{align}\label{eq:formula_etaprime}
\eta'(m)h_x(\eta(m),m)+h_m(\eta(m),m)=0.
\end{align}
By \eqref{HJBbench} and Assumption \ref{ass:A2},
\begin{align*}
h_x(x,m)=V^m_x(x) - \frac{\mu^{\prime}(x)}{r} \geq 1 -\frac{\mu^{\prime}(x)}{r} > 0.
\end{align*}
On the other hand, the mapping $m \mapsto V^m(x)$ is decreasing on $(\delta,\overline x)$ by Step 3, so that $h_m(\eta(m),m) \leq 0$ and thus $\eta'(m) \geq 0$ for all $m \in (\delta,\overline x)$. Moreover, $\eta$ must be increasing on $(\delta,\overline x)$. Otherwise, we would have $\eta(m)=x$ for all $m$ on some non-degenerate interval, so that $m \mapsto V^m (x)$ would also be constant on that interval, a contradiction. Finally, define $\underline\eta:=\lim_{m \rightarrow \delta}\eta(m)$, which exists by monotonicity. By continuity of $N$, we have $N(\underline \eta, \delta)=0$. Together with \eqref{eq:deriv_g_neq0}, this implies that $\delta \in H$, and thus $\delta=0$ as $H$ is open in $[0, \overline x]$. From \eqref{eq:formula_etaprime}, we deduce that $\eta'$ remains bounded at $0$, so that $\eta$ is $\CC^1$ and increasing on $[0,\overline x]$. The claim follows. Hence the result. \hfill $\blacksquare$

\bigskip

\noindent \textbf{Proof of Proposition \ref{skob}.} We will assume in what follows that $\mu$ and $\sigma$ are defined and globally Lipschitz on $\RR$; this does not affect the result of Proposition \ref{skob} as the solutions are considered only up to the first hitting time of $0$.
We start with the following existence result.

\begin{applem} \label{step1}
Let $l$ be a differentiable function defined on $(-\infty,\overline m]$ such that $l(\overline m)=\overline m$ and $|l'(m)| < c < 1$ for all $m \in (-\infty,\overline m]$. Then$,$ for every initial condition $(x,m) \in \RR^2$ with $x\geq m,$ there exists a unique strong solution
$(X,M,L)$ to the reflected sde
\begin{align}
X_{0-}&=x, \; M_{0-}=m, \; L_{0-}=0, \label{sko00l}
\\
\mathrm dX_t &= \mu(X_t) \, \mathrm dt + \sigma (X_t) \, \mathrm dB_t - \mathrm dL_t,\; M_t = m \wedge \inf_{0 \leq s \leq t} X_s, \quad t \geq 0, \label{sko0l}
\\
X_t & \in [M_t, l(M_t)]\cap (- \infty,\overline m] \;\; \mathbb P\text{-a.s.}, \quad t \geq 0, \label{sko1l}
\\
\int_{[0,t]}&\, \ind_{\{X_s < l(M_s) \}} \, \mathrm dL_s = 0 \;\; \mathbb P\text{-a.s.}, \quad  t\geq 0. \label{sko2l}
\end{align}
\end{applem}

\vskip 3mm

\noindent \textbf{Proof.} First, notice that, as $l$ is Lipschitz with constant $c<1$, we have $l(m)>m$ for all $m\in (-\infty,\overline m]$.
Conditions \eqref{sko1l}--\eqref{sko2l} imply that for any solution
\begin{align*}
L_0 & = (x-\overline m)\ind_{m> \overline m}+ [x-l(m)]\ind_{m\leq \overline m \text{ and } x \geq l(m)},
\\
X_0 &= \overline m \ind_{m > \overline m} + l(m)\ind_{\{m\leq \overline m \text{ and } x \geq l(m)\}},
\\
M_0 &=\overline m \ind_{m>\overline m} + m \ind_{m \leq \overline m},
\end{align*}
and $(X,M,L)$ are continuous except for a possible deterministic jump at time 0. Thus we may with no loss of generality assume that the initial condition is $(x,m)=(l(m),m)$ for some $m\in (-\infty,\overline m]$, as the general solution can be obtained by adding the initial deterministic jump. To simplify notation, we will only write the proof in the case $x=m=\overline m$, the other cases being similar.

The proof follows usual arguments. We first show that, for $T >0$ small enough, there exists on $[0,T]$ a unique continuous solution to   (\ref{sko0l})--(\ref{sko2l}) with $X_0 = M_0= \overline m$. For $T >0$, let $\mathcal X$ be the space of continuous $\mathbb F$-adapted processes $X:=(X_t)_{0 \leq t \leq T}$ such that $\EE \hskip 0.15mm \big[ \sup_{0\leq t \leq T}| X_t |^2\big]<\infty$, which, endowed with the norm $\| X \| =\sqrt{ \EE \hskip 0.15mm \big[\sup_{0\le t \le T}| X_t |^2\big]}$, is a Banach space. We will apply the contraction mapping theorem to the mapping $\Phi$ defined on $\cal{X}$ by
\begin{align*}
\Phi(X)_t : = \overline m+ \int_0^t \mu(X_s) \, \mathrm ds  + \int_0^t \sigma(X_s) \, \mathrm dB_s  -  L(X)_t, \quad t\in [0,T],
\end{align*}
where
\begin{align*}
L(X)_t : = \sup_{0 \leq s \leq t}\left[  \overline m+ \int_0^s \mu(X_u) \, \mathrm du  +  \int_0^s \sigma(X_u) \, \mathrm dB_u  -  l (M^X_s )\right]^+ \hskip -1.5mm, \quad t\in [0,T],
\end{align*}
where $M^X_s : =\inf_{u\in[0,s]} X_u$ for all $s \in [0,t]$. Notice that the process $\Phi(X)$ is well-defined, $\mathbb F$-adapted, and continuous. Let us show that $\Phi$ takes values in $\mathcal X$ and is Lipschitz. Consider $X $ and $Y$ two elements of ${\cal X}$. Using that $\mu$ and $\sigma$ are Lipschitz along with Doob's inequality, we obtain that, for some positive constant $K$,
\begin{align}
\sqrt{ \EE \! \left[\sup_{0\leq t \leq T}\left|\int_0^t [\mu(X_s) - \mu(Y_s)]\,\mathrm ds \right|^2 \right]}\, &+ \sqrt{\EE \! \left[\sup_{0\leq t \leq T} \left|\int_0^t [\sigma(X_s) - \sigma(Y_s)]\, \mathrm dB_s \right|^2 \right]}\nonumber
\\
&\leq  (K T+ 2 K \sqrt{T}) \, \sqrt{ \EE \! \left[ \sup_{0\leq s \leq T}|X_t - Y_t|^2 \right]} \nonumber
\\
&<\infty. \label{reflex2}
\end{align}
Letting $M_s^X:=\inf_{u\in[0,s]}X_u$ and $M_s^Y:=\inf_{u\in[0,s]}Y_u$, one can easily check that, for $t\in [0,T]$,
\begin{align*}
|L(X)_t - L(Y)_t| \leq \sup_{0 \leq s \leq t} &\, \Bigg\{ \left[ \overline m+ \int_0^s \mu(X_u) \, \mathrm du + \int_0^s \sigma(X_u) \, \mathrm dB_u - l (M^X_s )\right]^+
\\
& - \left[  \overline m+ \int_0^s \mu(Y_u) \, \mathrm du + \int_0^s \sigma(Y_u) \, \mathrm dB_u - l (M^Y_s )\right]^+ \Bigg\},
\end{align*}
which, because the mapping $x\mapsto x^+$ is 1-Lipschitz, leads to
\begin{align*}
&|L(X)_t - L(Y)_t|
\\
& \leq \sup_{0\leq s \leq t} \left \{ \left| \int_0^s [\mu(X_u) - \mu(Y_u)] \,\mathrm du \right| + \left|\int_0^s [\sigma(X_u) - \sigma(Y_u)] \,\mathrm dB_u \right| + \left|l(M_s^X) - l(M_s^Y)\right|\right \} \!,
\end{align*}
so that
\begin{align}
\| L(X) - L(Y) \| &\leq  \sqrt{ \EE \! \left[\sup_{0\leq t \leq T}\left|\int_0^t [\mu(X_s) - \mu(Y_s)]\,\mathrm ds \right|^2 \right]} \nonumber
\\
&  \quad +\sqrt{\EE \! \left[\sup_{0\leq t \leq T} \left|\int_0^t [\sigma(X_s) - \sigma(Y_s)]\, \mathrm dB_s \right|^2\right]}  \nonumber
\\
& \quad + \sqrt{ \EE\! \left[\sup_{0\leq t \leq T}\left|l(M_t^X)- l(M_t^Y)\right|^2\right]}. \label{reflex3}
\end{align}
Because $l$ is $c$-Lipschitz, we have
\begin{align}
\sqrt{{\EE} \!\left [   \sup_{0 \leq t\leq T}   \left | l(M_t^X) - l(M_t^Y)\right |^2  \right ]} & \leq c\, \sqrt{{\EE} \left [   \sup_{0 \leq t\leq T}   \left |M_t^X - M_t^Y\right |^2  \right ]} \nonumber
\\
& \leq c\,  \sqrt{{\EE} \left [   \sup_{0 \leq t\leq T}   \left | X_t - Y_t\right |^2  \right ]}  \nonumber
\\
& <\infty. \label{reflex4}
\end{align}
Using the triangle inequality, it follows from (\ref{reflex2})--(\ref{reflex4}) that
\begin{align}
\| \Phi(X) - \Phi(Y) \| &\leq  \sqrt{ \EE \! \left[\sup_{0\leq t \leq T}\left|\int_0^t [\mu(X_s) - \mu(Y_s)]\,\mathrm ds \right|^2 \right]} \nonumber
\\
&  \quad +\sqrt{\EE \! \left[\sup_{0\leq t \leq T} \left|\int_0^t [\sigma(X_s) - \sigma(Y_s)]\, \mathrm dB_s \right|^2\right]} \nonumber
\\
& \quad  + \| L(X) - L(Y) \| \nonumber
\\
&\leq  [2(K T+ 2 K \sqrt{T}) + c] \| X - Y \| \nonumber
\\
&<\infty. \label{lip}
\end{align}
The proof that $\vert\vert \Phi (X) \vert\vert < \infty$ for $x\in \mathcal X$ follows along the same lines as the proof of \eqref{lip} and is thus omitted. Therefore $\Phi$ takes values in $\mathcal X$ and is Lipschitz with coefficient $2(K T+ 2 K \sqrt{T}) +c$. For $T$ small enough, we have $2(K T+ 2 K \sqrt{T}) + c <1$ and $\Phi$ is a contraction; by the contraction mapping theorem, we get the existence and the uniqueness in ${\cal X}$ of a fixed point $X$ of $\Phi$. From the definition of $\Phi$, the triple $(X,L(X),M^X)$ satisfies (\ref{sko0l})--(\ref{sko2l}) on $[0,T]$ with initial condition $X_0=M_0=\overline m$.

Finally, we show that any strong solution $X$ to (\ref{sko0l})--(\ref{sko2l}) on $[0,T]$ necessarily belongs to $\cal X$. The proof again follows usual arguments. We consider $f_n(t) : = \EE \hskip 0.15mm \big[\sup_{0\leq s \leq t \wedge T_n} |X_s| ^2\big]$ with $T_n :=\inf \{ u \geq 0 : |X_u| >n\}$ for all $n \geq 1$, and, proceeding analogously as above, we find two constants $c_1$ and $c_2$ independent of $n$ such that $f_n(t) \leq c_1 + c_2\int_0^t f_n(s) \, \mathrm ds$ for all $0 \leq t\leq T$. By Gronwall's lemma, we obtain $\EE \hskip 0.15mm \big[\sup_{0\leq t \leq T \wedge T_n} |X_s|^2\big] \leq C$ for some constant $C$ independent of $n$. Letting $n \to \infty$, we conclude that $X \in {\cal X}$.

To conclude for all $T>0$, we consider a subdivision of $[0,T]$ in $n$ subintervals where $n$ is large enough to ensure the existence and uniqueness of the solution to (\ref{sko0l})--(\ref{sko2l}) on the intervals $([{k \over n}\, T, {k+ 1 \over n}\, T])_{0 \leq k \leq n-1}$. This only requires extending the previous proof to a random initial condition, which is direct by adapting the notation. The result follows. \hfill $\blacksquare$

\bigskip

We now return to the proof of Proposition \ref{skob}. We shall only consider the case $\tau=0$ as the proof, including Lemma \ref{step1}, can be easily extended to an arbitrary stopping time $\tau$. The function $b$ is defined on $[0,\overline x]$, with $b(0)=x^0$. We extend $b$ to $(- \infty, \overline m]$ by defining $b(m): =x^0-m$ for $m<0$ to maintain the ratio $\frac{b(m)+m}{2}$ constant for $m \le 0$. According to Proposition \ref{lemedo}, the function $b$ is $\CC^1$ on $(0,\overline m]$, with $b'(\overline m) <1$. Therefore, there exists some $m_0 \in (0, \overline m)$ such that $b'(m) <1$ for $m \in [m_0, \overline m]$. Moreover, observe that the mapping $m \mapsto b(m)-m$ is uniformly bounded below by some constant $\delta>0$ on $(-\infty,m_0]$. Now, let $(x,m) \in \CJ$. If $m \geq m_0$, applying Lemma \ref{step1} to some function $l$ that coincides with $b$ on $[m_0, \overline m]$, there exists a unique strong solution $(\overline X, \overline M,\overline L)$ to
\begin{align*}
\overline X_{0-}& =x, \;\overline M_{0-}=m, \; \overline L_{0-}=0,
\\
\mathrm d \overline X_t &= \mu(\overline X_t) \, \mathrm dt + \sigma (\overline X_t) \, \mathrm dB_t - \mathrm d\overline L_t,\; \overline M_t=m \wedge \inf_{0 \leq s \leq t} \overline X_s,\quad 0\le t \le \gamma_0 ,
\\
\overline X_t &\in [ \overline M_t, b( \overline  M_t)] \cap[m_0,\overline m] \;\; \mathbb P\text{-a.s.}, \quad  0\le t \le \gamma_0 ,
\\
\int_{[0,t]}&\, \ind_{\{\overline X_s < b( \overline M_s) \}} \, \mathrm d \overline L_s = 0 \;\; \mathbb P\text{-a.s.}, \quad  0\le t \le \gamma_0,
\end{align*}
where $\gamma_0:=\inf \hskip 0.5mm \{ t \ge 0: \overline X_t \leq m_0\}$. If $m <m_0$, we simply define $\gamma_0: =0$, $\overline X_0:=x \wedge b(m)$, $\overline M_0:=m$, and $\overline L_0:=x-x\wedge b(m)$.

Now, set $b_0:=b(m_0)$ and consider the sequence of pairs $(X^{(k)}, L^{(k)})_{k \geq 0}$ recursively defined by $X^{(0)}_t := \overline X_t $ and $L^{(0)}_t := \overline L_t$ for any time $t$ in the random interval $[0,\gamma_0]$ and, for $k \geq 1$,
\begin{align*}
X_t^{(k)} &:=X_{\gamma_{k-1}}^{(k-1)} + \int_{\gamma_{k-1}}^t \mu(X_s^{(k)}) \, \mathrm ds  +  \int_{\gamma_{k-1}}^t \sigma(X_s^{(k)}) \,\mathrm  dB_s -  L_t^{(k)},
\\
L_t^{(k)} & : = \sup_{\gamma_{k-1} \leq s \leq t}\left[ X_{\gamma_{k-1}}^{(k-1)}+\int_{\gamma_{k-1}}^s \mu(X_u^{(k)}) \, \mathrm du +\int_{\gamma_{k-1}} ^s \sigma(X_u^{(k)}) \, \mathrm dB_u -  b_k\right]^+
\end{align*}
for any time $t$ in the random interval $(\gamma_{k-1}, \gamma_{k}]$, where
\begin{align*}
\gamma_{k}:=\inf\left\{t \ge \gamma_{k-1}: M_t^{(k)}<m_k \text{ and }X^{(k)}_t=\frac{b(M^{(k)}_t)+M^{(k)}_t}{2} \right\}\!,
\end{align*}
for
\begin{align*}
m_k:=M^{(k-1)}_{\gamma_{k-1}}, \; b_k:=b(m_k), \text{ and } M_t^{(k)}:=m_k \wedge \inf_{\gamma_{k-1} <s \le t} X^{(k)}_s.
\end{align*}
That is, by definition, $(X^{(k)}_t)_{\gamma_{k-1} < t \leq \gamma_k}$ is the unique strong solution to an sde reflected at the fixed level $b_k$. For $k=0$, we let $(X_t,L_t) : = (X_t^{(0)},L_t^{(0)})$ for all $t \in [0, \gamma_{ 0}]$, and, for each $k\geq 1$, we let $(X_t,L_t) := (X_t^{(k)},L_{ \gamma_{k-1}} +L_t^{(k)})$ for all $t \in (\gamma_{k-1}, \gamma_{k}]$. By construction, the pair $(X_t,L_t)$ is defined on $[0, \lim_{k\to \infty} \gamma_k)$ and satisfies (\ref{sko0})--(\ref{sko2}) on that time interval.

To conclude the proof, it thus only remains to show that $\gamma_k \to \infty$ $\mathbb P$-a.s. Consider the sequence of random variables
\begin{align*}
Z_k : =\frac{b_k+m_k}{2} \, \ind_{\gamma_k<\infty}, \quad k \geq 0,
\end{align*}
which notably satisfies $X_{\gamma_k}=Z_k$ on $\{\gamma_k<\infty\}$. The sequence $(Z_k)_{k\geq 0}$ is nonincreasing and bounded below by $0$, and we have $Z_k \geq \frac{x^0}{2}$ on $\{\gamma_k<\infty\}$ by construction of the extension of the function $b$ to $(-\infty,0)$. Therefore, the sequence $(Z_k)_{k\geq 0}$ $\mathbb P$-a.s. converges to a nonnegative random variable $Z$ that satisfies $Z\ge \frac{x^0}{2}> \frac{\delta}{2}$ on $\{\lim _{k\to \infty} \gamma_k<\infty\}$. We have
\begin{align}
{\PP}\left[ \gamma_{k+1}-\gamma_k>c \! \mid \!  \CF_{\gamma_k}\right]\ind_{\gamma_k<\infty} \geq \PP \left[T_k(\delta) >c \! \mid \! \CF_{\gamma_k} \right]\ind_{\gamma_k<\infty} \label{am}
\end{align}
for some positive constant $c$, where
\begin{align*}
T_k(\delta) : =\inf\left\{ t \ge 0: X_{\gamma_k+t} \notin \hskip -0.2mm \left[Z_k-{\textstyle\frac{\delta}{4}},Z_k+{\textstyle \frac{\delta}{4}} \right]\right\} \ind_{\gamma_k<\infty}.
\end{align*}
By construction, on $\{\gamma_k<\infty\}$, the process $X$ is a solution on the random time interval $[\gamma_k,T_k(\delta)]$ of the uncontrolled sde
\begin{align} \label{SDEA}
\mathrm dX_t= \mu(X_t)\, \mathrm dt+\sigma(X_t)\, \mathrm dB_t.
\end{align}
Therefore, from, the strong Markov property,
\begin{align}
\PP \left[T_k(\delta) >c \! \mid \! \CF_{\gamma_k} \right]\ind_{\gamma_k<\infty}=\widehat\PP_{Z_k}\hskip 0.3mm[T(\delta)>c]\,\ind_{\gamma_k<\infty}, \label{stram}
\end{align}
where $\widehat\PP_x$ denotes the distribution of the unique solution $X$ to \eqref{SDEA} with initial condition $x$ at time $0$ and
\begin{align*}
T(\delta) :=\inf \left\{t \geq 0: X_t \notin \hskip -0.2mm \left[X_0- {\textstyle{\frac{\delta}{4}, X_0+\frac{\delta}{4}}}\right]\right\}.
\end{align*}
By Markov's inequality,
\begin{align}
\widehat\PP_{Z_k}\hskip 0.3mm[T(\delta) >c]\,\ind_{\gamma_k<\infty} &=\left\{1-\widehat\PP_{Z_k}\big[\mathrm e^{-rT(\delta)} \geq \mathrm e^{-rc} \big]\right \}\ind_{\gamma_k< \infty} \nonumber
\\
&\ge \left\{1-\mathrm e^{rc}\,{\widehat\EE}_{Z_k}\big[\mathrm e^{-rT(\delta)} \big]\right\}\ind_{ \gamma _k <\infty}. \label{gram}
\end{align}
Now, consider the standard exit-time problem for a one-dimensional sde,
\begin{align*}
u(x,a,b): ={\widehat\EE}_x\big[ \mathrm e^{-rT_{a,b}}\big] \text{ for } T_{a,b}=\inf \hskip 0.5mm \{t \ge 0 : X_t \notin [a,b] \}.
\end{align*}
In line with \cite[Chapter 5, \S5]{KaratzasShreve}, one can check that the function $u$ is jointly continuous for every interval $[a,b]$ strictly included in the state space $\RR$, as per our initial convention in this proof. Therefore,
\begin{align*}
{\widehat\EE}_{Z_k}\big[\mathrm e^{-rT(\delta)} \big]=u (Z_k,Z_k-{\textstyle \frac{\delta}{4}},Z_k+{\textstyle\frac{\delta}{4}}).
\end{align*}
For each $k\geq 0$, we have $Z_k \in [\frac{x^0}{2}, b_0]$ on $\{\gamma_k<\infty\}$ and thus
\begin{align}
{\widehat\EE}_{Z_k}\big[\mathrm e^{-rT(\delta)} \big] \le \sup_{\frac{x^0}{2} \leq z \leq b_0}u(z,z-{\textstyle{\frac{\delta}{4} ,z+\frac{\delta}{4}}}): =\overline u <1. \label{pic}
\end{align}
Choosing $c \in \big[0,\frac{1}{r}\ln\big( \frac{1}{\overline u}\big) \big] $, we conclude from \eqref{am}, \eqref{stram}--\eqref{gram}, and \eqref{pic} that
\begin{align*}
{\PP}\left[ \gamma_{k+1}-\gamma_k>c \! \mid \!  \CF_{\gamma_k}\right] \ge (1- \mathrm e^{rc}\overline u) >0
\end{align*}
on $\{\gamma_k<\infty\}$ for all $k \geq 0$. As a result, the series of general term ${\PP}\left[ \gamma_{k+1}-\gamma_k>c \! \mid \!  \CF_{\gamma_k} \right]$ diverges on $\{ \lim_{k \to \infty} \gamma_k < \infty\}$. By L\'evy's extension of the Borel--Cantelli lemma (\cite[Theorem 12.15(b)]{Williams}), we conclude that $\gamma_{k+1}-\gamma_k>c$ infinitely often on $\{ \lim_{k \to \infty} \gamma_k < \infty\}$ $\mathbb P$-a.s., and, hence, $\gamma_k \to \infty$  $\mathbb P$-a.s. Hence the result. \hfill $\blacksquare$

\subsection{On $\boldsymbol{u^* <\overline m}$ versus $\boldsymbol{u^* >\overline m}$} \label{Apptwocases}

In this appendix, we illustrate by means of an example that the two scenarios $u^* < \overline m$ and $u^* >\overline m$ are indeed relevant for our analysis.

Suppose that the uncontrolled reserve process $X$ is an arithmetic Brownian motion with drift $\mu>0$ and volatility $\sigma>0$, and that $U$ is, as in Example \ref{example}, the value function of a downgraded extraction problem whose uncontrolled reserve process $\underline X$ is also represented by an arithmetic Brownian motion with volatility $\sigma$, but with a lower drift $\underline \mu \in (0,\mu)$. For $L \in \mathbb L$, define the corresponding controlled process $\underline X^L$ as the unique strong solution of
\begin{align*}
\mathrm d\underline X ^L_t= \underline \mu \, \mathrm dt +\sigma \, \mathrm dB_t - \mathrm dL_t, \quad \underline X_{0-}=z,
\end{align*}
where $B$ is the same Brownian motion driving $X$, and define
\begin{align}
U(z): =\sup_{L \in \underline \CA(0)} \EE_z\!\left[\int_{[0,\underline \tau_0]} \mathrm  e^{-rt} \, \mathrm dL_t \right]\hskip -1mm, \quad z \geq 0,
\label{Garance}
\end{align}
where $\underline \CA(0): =\{ L \in \LL :(\underline X_t^L)^+-(\Delta L)_t\geq 0 \text{ for all } t\geq 0\}$ and $\underline \tau_0: =\inf \hskip 0.5mm \{ t \geq 0 : \underline X_t^L= 0\}$. Then, as explained in Example \ref{example}, $U$ satisfies Assumption \ref{ass:A6}, and Assumption \ref{ass:A2} is also satisfied as $rU(0)=0 < \mu$. It is well-known (see, for instance, \cite{Jeanblanc}) that the optimal extraction threshold for
\eqref{Garance} is
\begin{align*}
u^*:=\frac{2}{\underline \beta^+ - \underline \beta^-}\,\ln \!\left(\frac{-\underline \beta^-}{\underline \beta^+} \right) \hskip -1mm,
\end{align*}
where $\underline \beta^-<0<\underline \beta^+$ are the two roots of the quadratic equation $\frac{\sigma^2}{2}\beta^2 + \underline \mu \beta - r=0$. Similarly, because $U(0)=0$, the optimal extraction threshold for \eqref{aux} with $m=0$ is given by
\begin{align*}
x^0:=\frac{2}{ \beta^+ - \beta^-}\,\ln \!\left(\frac{- \beta^-}{ \beta^+} \right) \hskip -1mm,
\end{align*}
where $\beta^-<0<\beta^+$ are the two roots of the quadratic equation $\frac{\sigma^2}{2}\beta^2 + \mu \beta - r=0$. Holding $r$ and $\sigma$ fixed, it is known that $x^0$ is a single-peaked function of $\mu \in (0, \infty)$, with maximum $\mu^* >0$ (\cite[Proposition 10]{RochetVilleneuve}). It is then useful to distinguish two polar cases, where the comparison between $u^*$ and $x^0$ is unambiguous.

First, if $\underline \mu < \mu \leq \mu^*$, then $u^* < x^0$. In this case, as $x^0 < \overline m$ by Proposition \ref{lemedo}, we necessarily have $u^* < \overline m$.

Second, if $\mu^* \leq \underline \mu < \mu$, then $u^* > x^0$. Then, by continuity and monotonicity of $\eta,$ there exists $c \in (0, \overline x)$ such that $u^* > \eta(c)$. We claim that, in this case, if the density $f$ of $Y$ has support $[0, \overline x]$ and is mostly concentrated on $[0, c]$, then $u^* >\overline m$. Suppose indeed that $f \leq \varepsilon$ over $[c, \overline x]$ for some small $\varepsilon >0$. Then
\begin{align}
H(m) = {f(m) \over F(m)} \leq \frac{\varepsilon}{1 - \varepsilon (\overline x - m)} \text{ for all } m\in [c, \overline x]. \label{bound H}
\end{align}
Now, recall that $u^* \leq \overline x$ and $\eta(u^*)\geq u^*$, so that $\eta(c)<u^*\leq \eta(u^*)$ implies $c<u^*$ as $\eta$ is increasing. As the mapping $(x,m) \mapsto E (x,m) =\frac {H(m) N(x,m)}{ G(x) D(x,m)}$ is continuous on ${\cal K} : = \{(x,m) \in \mathcal J: x \leq \overline x \text{ and } c \leq m \leq \overline x \}$ as $c >0$, we can thus by \eqref{bound H} choose $\varepsilon$ small enough that
\begin{align*}
\overline E_{\mathcal K} :=\sup_{(x,m) \in \cal{K}}  E(x,m) < \frac{u^*-\eta(c)}{u^*-c}.
\end{align*}
By Proposition \ref{lemedo}, $ b(c ) < \eta (c)$ as $c>0$, and, by assumption, $ \eta (c)< u^*$. Because $b$ is solution to the Cauchy problem \eqref{ode1} and $\overline y=\overline x$, we have
\begin{align*}
b(u^*)=b(c)+ \int_{c}^{u^*} E(b(s),s)\, \mathrm ds \leq \eta(c)+ \overline E _{\cal K}(u^*-c) < u^*.
\end{align*}
Thus, as $b(0)=x^0>0$ and $b$ is continuous, we conclude from the intermediate value theorem that $\overline m <u^*$. The claim follows.

\end{document}